\documentclass[12pt,a4paper,twoside]{article}

\title{\sc Generalized Electro-Magneto Statics in Nonsmooth Exterior Domains}
\def\shorttitle{Generalized Electro-Magneto Statics}
\def\pauthor{Dirk Pauly}

\def\mylabelonoff{off}
\def\allowdisbrk{no}

\usepackage{a4,exscale,ifthen,amsfonts,amssymb,amsmath,amscd,graphicx}
\usepackage[english]{babel}
\usepackage[mathscr]{eucal}

\author{{\sf\pauthor}}
\numberwithin{equation}{section}

\newenvironment{acknow}{{\vspace*{1cm}\noindent\bf Acknowledgements }}{}
\newcommand{\bewboxw}{\mbox{}\hfill $\square$ \\}

\newenvironment{proof}{{\noindent\bf Proof }}{\bewboxw}

\newcommand{\keywords}[1]{{\noindent\bf Key Words }#1}
\newcommand{\amsclass}[1]{{\noindent\bf AMS MSC-Classifications }#1}

\ifthenelse{\equal{\mylabelonoff}{on}}
{\newcommand{\mylabel}[1]{\label{#1}\fbox{{\rm #1}}}}{\newcommand{\mylabel}[1]{\label{#1}\makebox[0mm][]{}}}
\ifthenelse{\equal{\allowdisbrk}{yes}}
{\allowdisplaybreaks}{}

\newcommand{\paper}[7]{\bibitem{#1} #2, `#3', {\it #4}, #5, (#6), #7.}

\newcommand{\dissavail}[6]{\bibitem{#1} #2, `#3', {\sf Dissertation}, #4, (#5), available from {\tt #6}.}

\usepackage{amsfonts,amssymb,amsmath,amscd}
\usepackage[mathscr]{eucal}

\newcommand{\schluss}{\ifodd\value{page}\newpage\thispagestyle{empty}\makebox[0mm][]{}\color{sehrhell}.\fi\end{document}}
\newcommand{\non}{\nonumber}

\newcommand{\ol}[1]{\overline{#1}}
\newcommand{\ul}[1]{\underline{#1}}

\newcommand{\ub}{\underbrace}
\newcommand{\qtext}[1]{\quad\text{#1}\quad}
\newcommand{\qqtext}[1]{\qquad\text{#1}\qquad}
\newcommand{\peh}{\frac{1}{2}}
\newcommand{\cast}{\circledast}
\newcommand{\Nh}{\frac{N}{2}}
\newcommand{\meh}{{-\peh}}
\newcommand{\me}{{-1}}

\newcommand{\intom}{\int_\Omega}

\newcommand{\om}{{\Omega}}

\newcommand{\omq}{\ol{\om}}

\newcommand{\dom}{{\p\om}}

\newcommand{\ohne}{\setminus}

\newcommand{\eps}{\varepsilon}

\newcommand{\epsd}{\hat{\eps}}

\newcommand{\mud}{\hat{\mu}}

\newcommand{\homg}[3]{{{}^{#3}\mathrm{h}^{#1}_{#2}}}

\newcommand{\EH}{(E,H)}

\newcommand{\EHs}{(\tilde{E},\tilde{H})}

\newcommand{\FG}{(F,G)}

\newcommand{\bon}[1]{\overset{\circ}{b}{}^{#1}}

\newcommand{\To}{\longrightarrow}

\newcommand{\equi}{\Leftrightarrow}

\newcommand{\restr}[2]{\left.#1 \right|_{#2}}

\newcommand{\Abb}[5]{\begin{array}{ccccc}#1&:&#2&\longrightarrow&#3\\{}&{}&#4&\longmapsto&#5\end{array}}

\newcommand{\bds}{\begin{displaystyle}}
\newcommand{\eds}{\end{displaystyle}}
\newcommand{\beq}{\begin{equation}}
\newcommand{\eeq}{\end{equation}}
\newcommand{\beqa}{\begin{eqnarray}}
\newcommand{\eeqa}{\end{eqnarray}}
\newcommand{\beqanon}{\begin{eqnarray*}}
\newcommand{\eeqanon}{\end{eqnarray*}}
\newcommand{\bamath}{$$\begin{array}}
\newcommand{\eamath}{\end{array}$$}
\newcommand{\ba}{\begin{array}}
\newcommand{\ea}{\end{array}}
\newcommand{\bc}{\begin{center}}
\newcommand{\ec}{\end{center}}
\newcommand{\bmini}{\begin{minipage}}
\newcommand{\emini}{\end{minipage}}

\newtheorem{lem}{Lemma}[section]
\newtheorem{theo}[lem]{Theorem}
\newtheorem{cor}[lem]{Corollary}
\newtheorem{defini}[lem]{Definition}
\newtheorem{rem}[lem]{Remark}

\newcommand{\set}[2]{\{#1 \;\text{\bf :}\;  #2\}}

\newcommand{\setb}[2]{\big\{#1 \;\text{\bf :}\;  #2\big\}}

\newcommand{\SN}{S^{N-1}}

\DeclareMathOperator{\B}{B}

\newcommand{\pb}[2]{\overset{#2}{\B}{}^{#1}}

\newcommand{\bqom}{\B^q(\Omega)}
\newcommand{\bonqom}{\pb{q}{\circ}(\Omega)}
\newcommand{\bqpeom}{\B^{q+1}(\Omega)}

\newcommand{\rz}{\mathbb{R}}

\newcommand{\nz}{\mathbb{N}}
\newcommand{\zz}{\mathbb{Z}}

\newcommand{\cz}{\mathbb{C}}

\newcommand{\nzn}{{\nz_0}}
\newcommand{\rd}{{\rz^3}}
\newcommand{\rN}{{\rz^N}}

\newcommand{\pI}{\mathbb{I}}

\newcommand{\cIe}{\sideset{_1}{}\cI}
\newcommand{\cJe}{\sideset{_1}{}\cJ}
\newcommand{\cIz}{\sideset{_2}{}\cI}
\newcommand{\cJz}{\sideset{_2}{}\cJ}
\newcommand{\cIj}{\sideset{_j}{}\cI}
\newcommand{\cJj}{\sideset{_j}{}\cJ}
\newcommand{\cIq}{\cI^q}
\newcommand{\cJq}{\cJ^q}
\newcommand{\cJqpe}{\cJ^{q+1}}
\newcommand{\cIb}{\bar{\cI}}
\newcommand{\cJb}{\bar{\cJ}}
\newcommand{\pcI}[3]{\cI^{#1,#2}_{#3}}
\newcommand{\pcJ}[3]{\cJ^{#1,#2}_{#3}}
\newcommand{\pcIb}[3]{\cIb^{#1,#2}_{#3}}
\newcommand{\pcJb}[3]{\cJb^{#1,#2}_{#3}}
\newcommand{\pcIq}[2]{\pcI{q}{#1}{#2}}

\newcommand{\pcJqpe}[2]{\pcJ{q+1}{#1}{#2}}
\newcommand{\pcIbq}[2]{\pcIb{q}{#1}{#2}}
\newcommand{\pcJbq}[2]{\pcJb{q}{#1}{#2}}
\newcommand{\pcJbqpe}[2]{\pcJb{q+1}{#1}{#2}}

\newcommand{\calO}{{\mathscr{O}}}

\newcommand{\calH}{{\mathscr{H}}}
\newcommand{\calP}{{\mathscr{P}}}

\newcommand{\calS}{{\mathscr{S}}}

\newcommand{\calR}{{\mathscr{R}}}
\DeclareMathOperator{\cR}{{\mathcal{R}}}
\newcommand{\calD}{{\mathscr{D}}}
\newcommand{\calM}{{\mathscr{M}}}

\newcommand{\calL}{{\mathcal{L}}}

\DeclareMathOperator{\cI}{{\mathscr{I}}}
\DeclareMathOperator{\cJ}{{\mathscr{J}}}

\newcommand{\zmat}[4]{\begin{bmatrix}#1&#2\\#3&#4\end{bmatrix}}

\DeclareMathOperator{\p}{\partial}

\DeclareMathOperator{\rhoh}{\check{\rho}}
\DeclareMathOperator{\tauh}{\check{\tau}}
\DeclareMathOperator{\loes}{\calL}
\DeclareMathOperator{\loesom}{\calL_\omega}
\DeclareMathOperator{\loesn}{\calL_0}
\DeclareMathOperator{\loc}{loc}
\DeclareMathOperator{\vox}{vox}

\DeclareMathOperator{\MAX}{{{\calM}{\rm ax}}}

\DeclareMathOperator{\statMax}{\mathfrak{Max}}

\DeclareMathOperator{\Eins}{\mathbf{1}}
\DeclareMathOperator{\Lin}{Lin}
\DeclareMathOperator{\id}{Id}

\DeclareMathOperator{\supp}{supp}

\DeclareMathOperator{\rot}{rot}
\DeclareMathOperator{\Rot}{Rot}
\DeclareMathOperator{\ROT}{ROT}
\DeclareMathOperator{\pdiv}{div}
\DeclareMathOperator{\Div}{Div}
\DeclareMathOperator{\DIV}{DIV}
\DeclareMathOperator{\crot}{\mathscr{ROT}}
\DeclareMathOperator{\cdiv}{\mathscr{DIV}}

\DeclareMathOperator{\ie}{i}

\DeclareMathOperator{\pd}{d\hspace{-0.4ex}}

\DeclareMathOperator{\X}{X}
\DeclareMathOperator{\Y}{Y}

\DeclareMathOperator{\reg}{reg}

\DeclareMathOperator{\curl}{curl}
\DeclareMathOperator{\sgn}{sgn}
\DeclareMathOperator{\s}{\mathbf{s}}
\DeclareMathOperator{\he}{\mathbf{h}}
\DeclareMathOperator{\ei}{\mathbf{e}}
\DeclareMathOperator{\ci}{\mathbf{c}}
\DeclareMathOperator{\homd}{hom}

\newcommand{\pcoverset}[2]{\overset{#2}{\mathrm{C}}{}^{#1}}
\newcommand{\pc}[1]{\pcoverset{#1}{}}
\newcommand{\cn}[1]{\pcoverset{#1}{\circ}}
\newcommand{\cu}{\pc{\infty}}
\newcommand{\cun}{\cn{\infty}}

\newcommand{\lz}{\mathrm{L}^2}

\newcommand{\Lz}[1]{\lz_{#1}}

\newcommand{\Lzs}{\Lz{s}}

\newcommand{\h}[3]{\overset{#3}{\mathrm{H}}{}^{#1}_{#2}}

\newcommand{\hms}{\h{m}{s}{}}

\newcommand{\qc}[3]{\overset{#3}{\mathrm{C}}{}^{#1,#2}}
\newcommand{\cq}[1]{\qc{#1}{q}{}}
\newcommand{\cqn}[1]{\qc{#1}{q}{\circ}}
\newcommand{\cqu}{\cq{\infty}}
\newcommand{\cqun}{\cqn{\infty}}

\newcommand{\cqpe}[1]{\qc{#1}{q+1}{}}

\newcommand{\cqpeu}{\cqpe{\infty}}

\newcommand{\qlz}[1]{\mathrm{L}^{2,#1}}

\newcommand{\lzq}{\qlz{q}}

\newcommand{\qLz}[2]{\qlz{#1}_{#2}}
\newcommand{\qLzom}[2]{\qLz{#1}{#2}(\Omega)}
\newcommand{\Lzq}[1]{\qLz{q}{#1}}
\newcommand{\Lzqom}[1]{\Lzq{#1}(\Omega)}
\newcommand{\qLzs}[1]{\qLz{#1}{s}}

\newcommand{\qLzt}[1]{\qLz{#1}{t}}

\newcommand{\Lzqs}{\qLzs{q}}
\newcommand{\Lzqsom}{\Lzqs(\Omega)}
\newcommand{\Lzqt}{\qLzt{q}}
\newcommand{\Lzqtom}{\Lzqt(\Omega)}

\newcommand{\Lzqloc}{\Lzq{\loc}}

\newcommand{\Lzqvox}{\Lzq{\vox}}
\newcommand{\Lzqvoxom}{\Lzqvox(\Omega)}

\newcommand{\Lzqpe}[1]{\qLz{q+1}{#1}}
\newcommand{\Lzqpeom}[1]{\Lzqpe{#1}(\Omega)}
\newcommand{\Lzqpes}{\qLzs{q+1}}
\newcommand{\Lzqpesom}{\Lzqpes(\Omega)}
\newcommand{\Lzqpet}{\qLzt{q+1}}
\newcommand{\Lzqpetom}{\Lzqpet(\Omega)}

\newcommand{\qh}[4]{\overset{#4}{\mathrm{H}}{}^{#1,#2}_{#3}}
\newcommand{\qH}[4]{\overset{#4}{\mathbf{H}}{}^{#1,#2}_{#3}}

\newcommand{\hqms}{\qh{m}{q}{s}{}}

\newcommand{\hqpems}{\qh{m}{q+1}{s}{}}

\newcommand{\qHom}[4]{\qH{#1}{#2}{#3}{#4}(\Omega)}

\newcommand{\pr}[4]{{}_{#3}\overset{#4}{\mathrm{R}}{}^{#1}_{#2}}
\newcommand{\prq}[1]{\pr{q}{#1}{}{}}

\newcommand{\rqs}{\pr{q}{s}{}{}}

\newcommand{\rqt}{\pr{q}{t}{}{}}

\newcommand{\rqloc}{\pr{q}{\loc}{}{}}

\newcommand{\rqvox}{\pr{q}{\vox}{}{}}

\newcommand{\rqpen}[1]{\pr{q+1}{#1}{0}{}}
\newcommand{\rqsn}{\pr{q}{s}{0}{}}
\newcommand{\rqpesn}{\pr{q+1}{s}{0}{}}

\newcommand{\ronq}[1]{\pr{q}{#1}{}{\circ}}
\newcommand{\ronqpe}[1]{\pr{q+1}{#1}{}{\circ}}
\newcommand{\ronqs}{\pr{q}{s}{}{\circ}}
\newcommand{\ronqpes}{\pr{q+1}{s}{}{\circ}}
\newcommand{\ronqt}{\pr{q}{t}{}{\circ}}

\newcommand{\ronqn}[1]{\pr{q}{#1}{0}{\circ}}
\newcommand{\ronqpen}[1]{\pr{q+1}{#1}{0}{\circ}}
\newcommand{\ronqsn}{\pr{q}{s}{0}{\circ}}
\newcommand{\ronqpesn}{\pr{q+1}{s}{0}{\circ}}

\newcommand{\ronqpetn}{\pr{q+1}{t}{0}{\circ}}

\newcommand{\ronqpelocn}{\pr{q+1}{\loc}{0}{\circ}}

\newcommand{\rqom}[1]{\prq{#1}(\Omega)}

\newcommand{\rqsom}{\rqs(\Omega)}

\newcommand{\rqtom}{\rqt(\Omega)}

\newcommand{\rqlocomq}{\rqloc(\ol{\Omega})}

\newcommand{\rqvoxom}{\rqvox(\Omega)}

\newcommand{\rqpenom}[1]{\rqpen{#1}(\Omega)}

\newcommand{\rqpesnom}{\rqpesn(\Omega)}

\newcommand{\ronqom}[1]{\ronq{#1}(\Omega)}
\newcommand{\ronqpeom}[1]{\ronqpe{#1}(\Omega)}
\newcommand{\ronqsom}{\ronqs(\Omega)}
\newcommand{\ronqpesom}{\ronqpes(\Omega)}
\newcommand{\ronqtom}{\ronqt(\Omega)}

\newcommand{\ronqnom}[1]{\ronqn{#1}(\Omega)}
\newcommand{\ronqpenom}[1]{\ronqpen{#1}(\Omega)}
\newcommand{\ronqsnom}{\ronqsn(\Omega)}
\newcommand{\ronqpesnom}{\ronqpesn(\Omega)}

\newcommand{\ronqpetnom}{\ronqpetn(\Omega)}

\newcommand{\ronqpelocnom}{\ronqpelocn(\Omega)}

\newcommand{\pR}[4]{{}_{#3}\overset{#4}{\mathbf{R}}{}^{#1}_{#2}}

\newcommand{\Ronq}[1]{\pR{q}{#1}{}{\circ}}
\newcommand{\Ronqpe}[1]{\pR{q+1}{#1}{}{\circ}}

\newcommand{\Ronqom}[1]{\Ronq{#1}(\Omega)}
\newcommand{\Ronqpeom}[1]{\Ronqpe{#1}(\Omega)}

\newcommand{\pdi}[4]{{}_{#3}\overset{#4}{\mathrm{D}}{}^{#1}_{#2}}
\newcommand{\pdq}[1]{\pdi{q}{#1}{}{}}
\newcommand{\dqpe}[1]{\pdi{q+1}{#1}{}{}}
\newcommand{\dqs}{\pdi{q}{s}{}{}}
\newcommand{\dqpes}{\pdi{q+1}{s}{}{}}
\newcommand{\dqt}{\pdi{q}{t}{}{}}
\newcommand{\dqpet}{\pdi{q+1}{t}{}{}}

\newcommand{\dqvox}{\pdi{q}{\vox}{}{}}

\newcommand{\dqn}[1]{\pdi{q}{#1}{0}{}}

\newcommand{\dqsn}{\pdi{q}{s}{0}{}}

\newcommand{\dqtn}{\pdi{q}{t}{0}{}}

\newcommand{\dqlocn}{\pdi{q}{\loc}{0}{}}

\newcommand{\dqvoxn}{\pdi{q}{\vox}{0}{}}

\newcommand{\dqom}[1]{\pdq{#1}(\Omega)}
\newcommand{\dqpeom}[1]{\dqpe{#1}(\Omega)}
\newcommand{\dqsom}{\dqs(\Omega)}
\newcommand{\dqpesom}{\dqpes(\Omega)}
\newcommand{\dqtom}{\dqt(\Omega)}
\newcommand{\dqpetom}{\dqpet(\Omega)}

\newcommand{\dqvoxom}{\dqvox(\Omega)}

\newcommand{\dqnom}[1]{\dqn{#1}(\Omega)}

\newcommand{\dqsnom}{\dqsn(\Omega)}

\newcommand{\dqtnom}{\dqtn(\Omega)}

\newcommand{\dqlocnom}{\dqlocn(\Omega)}

\newcommand{\pDi}[4]{{}_{#3}\overset{#4}{\mathbf{D}}{}^{#1}_{#2}}
\newcommand{\pDq}[1]{\pDi{q}{#1}{}{}}
\newcommand{\Dqpe}[1]{\pDi{q+1}{#1}{}{}}

\newcommand{\Dqom}[1]{\pDq{#1}(\Omega)}
\newcommand{\Dqpeom}[1]{\Dqpe{#1}(\Omega)}

\newcommand{\dH}[3]{{}_{#3}\calH^{#1}_{#2}}
\newcommand{\dhqom}{\dH{q}{}{}(\Omega)}

\newcommand{\dhqpeom}{\dH{q+1}{}{}(\Omega)}

\newcommand{\dhqepsom}[1]{\dH{q}{#1}{\eps}(\Omega)}

\newcommand{\turmd}[5]{{{}^{#5}D^{#1,#2}_{#3,#4}}}
\newcommand{\turmr}[5]{{{}^{#5}R^{#1,#2}_{#3,#4}}}

\newcommand{\dqksmpm}{\turmd{q}{k}{\sigma}{m}{\pm}}

\newcommand{\dqkmesmpm}{\turmd{q}{k-1}{\sigma}{m}{\pm}}

\newcommand{\rqksmpm}{\turmr{q}{k}{\sigma}{m}{\pm}}
\newcommand{\rqpeksmpm}{\turmr{q+1}{k}{\sigma}{m}{\pm}}

\newcommand{\rqpekmesmpm}{\turmr{q+1}{k-1}{\sigma}{m}{\pm}}

\newcommand{\dqnsmpm}{\turmd{q}{0}{\sigma}{m}{\pm}}

\newcommand{\rqpensmpm}{\turmr{q+1}{0}{\sigma}{m}{\pm}}

\newcommand{\skp}[2]{\langle#1,#2\rangle}

\newcommand{\skpb}[2]{\big\langle#1,#2\big\rangle}

\newcommand{\bD}[3]{{}_{#3}\mathbb{D}^{#1}_{#2}}
\newcommand{\bR}[3]{{}_{#3}\overset{\circ}{\mathbb{R}}{}^{#1}_{#2}}
\newcommand{\bDqs}{\bD{q}{s}{0}}

\newcommand{\bRqpes}{\bR{q+1}{s}{0}}

\newcommand{\bDom}[3]{{}_{#3}\mathbb{D}^{#1}_{#2}(\Omega)}
\newcommand{\bRom}[3]{{}_{#3}\overset{\circ}{\mathbb{R}}{}^{#1}_{#2}(\Omega)}
\newcommand{\bDqsom}{\bD{q}{s}{0}(\Omega)}
\newcommand{\bRqsom}{\bR{q}{s}{0}(\Omega)}
\newcommand{\bDqlocom}{\bD{q}{\loc}{0}(\Omega)}

\newcommand{\bRqpesom}{\bR{q+1}{s}{0}(\Omega)}
\newcommand{\bRqpelocom}{\bR{q+1}{\loc}{0}(\Omega)}

\newcommand{\bWom}[2]{\mathbb{W}^{#1}_{#2}(\Omega)}
\newcommand{\bWqsom}{\bWom{q}{s}}
\newcommand{\bWqpesom}{\bWom{q+1}{s}}

\usepackage{algorithm,algorithmic}

\newcommand{\gt}{\gamma_\tau}

\newcommand{\chgt}{\check{\gamma}_\tau}

\newcommand{\xgn}{\gamma_n}
\newcommand{\xxrq}{\cR^q}

\newcommand{\paulydissregtheoaussen}{\cite[Satz 3.6]{paulydiss}}
\newcommand{\paulydissregtheoaussenzwei}{\cite[Satz 3.7]{paulydiss}}
\newcommand{\paulydisstopiso}{\cite[Satz 6.10]{paulydiss}}
\newcommand{\paulydissregcoraussen}{\cite[Korollar 3.8]{paulydiss}}
\newcommand{\paulydissabschnittsechsfuenf}{\cite[Abschnitt 6.5]{paulydiss}}

\newcommand{\paulytimeharmpartint}{\cite[(2.5)]{paulytimeharm}}
\newcommand{\paulytimeharmdefeps}{\cite[Definition 2.1 and 2.2]{paulytimeharm}}
\newcommand{\paulytimeharmcomrotdiv}{\cite[(2.10)]{paulytimeharm}}
\newcommand{\paulytimeharmmlcpdef}{\cite[Definition 2.4]{paulytimeharm}}
\newcommand{\paulytimeharmetadef}{\cite[(2.1), (2.2), (2.3)]{paulytimeharm}}
\newcommand{\paulytimeharmepsdef}{\cite[(2.6)]{paulytimeharm}}
\newcommand{\paulytimeharmlzzerl}{\cite[(2.7)]{paulytimeharm}}
\newcommand{\paulytimeharmsecboundary}{\cite[section 6]{paulytimeharm}}

\begin{document}

\date{2006}
\maketitle{}

\begin{abstract}
We develop a solution theory for a generalized electro-magneto static Maxwell system in
an exterior domain $\Omega\subset\rN$ with anisotropic coefficients converging at infinity with a rate
$r^{-\tau}$\,, $\tau>0$\,, towards the identity. Our main goal is to treat right hand side data from some
polynomially weighted Sobolev spaces and obtain solutions which are up to a finite sum of special generalized spherical harmonics
in another appropriately weighted Sobolev space.
As a byproduct we prove a generalized spherical harmonics expansion suited for Maxwell equations.
In particular, our solution theory will allow us to give meaning to higher powers of a special static solution operator.
Finally we show, how this weighted static solution theory can be extended to handle inhomogeneous boundary data as well.
This paper is the second one in a series of three papers, which will completely reveal the low frequency behavior
of solutions of the time-harmonic Maxwell equations.\\
\\
\amsclass{35Q60, 78A25, 78A30}
\end{abstract}

\tableofcontents

\section{Introduction}

In the bounded domain case it is just an easy exercise to show that the solution operator for the time-harmonic
Maxwell equations $\loesom$ is approximated by Neumann's series of the corresponding electro-magneto
static solution operator $\loes$ for small frequencies $\omega$\,, i.e.
$$\loesom=\omega^\me\Pi+\sum_{j=0}^\infty\omega^j\loes^{j+1}\Pi_{\reg}\qquad,$$
where $\Pi$ and $\Pi_{\reg}$ are projections onto irrotational and solenoidal fields.
In the case of an exterior domain (a domain with compact complement) this low
frequency asymptotic holds no longer true. We run into problems even if we formally want to
define higher powers of a static solution operator since the well known electro-magneto static solution theory developed
e.g. by Picard in \cite{potential,boundaryelectro,decomposition} treats data from a
polynomially weighted Sobolev space and yields solutions in a less weighted Sobolev space. In particular
in \cite{potential,boundaryelectro} we get from $\lz$-data $\Lz{-1}$-solutions
by decomposing $\lz$ into subspaces consisting of irrotational resp. solenoidal fields.
\big(Here for $s\in\rz$ we denote by $\Lzs$ the Hilbert space of all measurable fields $E$\,,
for which $\rho^sE$ is square integrable, where $\rho:=(1+r^2)^{1/2}$ and $r$ is the Euclidean norm in $\rN$\,.
Many of our notations have been previously used in \cite{paulytimeharm}.
For more details and the exact definitions we refer to this paper.\big)
In \cite{decomposition} we obtain from $\Lz{1}$-data $\lz$-solutions by a second order approach and elliptization of the
Maxwell system using Lax-Milgram's theorem. (The latter paper considers a more general non-linear case using a theorem
suited for monotone operators, but in the linear case it is just the Lax-Milgram theorem.)
Based on these known results we can only consider first and second powers of the solution operator of the static Maxwell system.
To overcome these limitations we have to develop an electro-magneto static solution theory,
which deals with arbitrarily weighted data and
describes the solutions in terms of their integrability properties, such that we are able to iterate this
static solution operator depending on the integrability of the data and therefore define a generalized Neumann
sum of the static solution operator.

In the case of Helmholtz' equation and the equations of linear elasticity theory in
an exterior domain $\Omega\subset\rN$\,, where comparable integrability problems for the static problem occur,
Weck and Witsch, \cite{complete} and \cite{linelae,linelaz}, respectively, have shown
that the time-harmonic solution operator is still approximated by a (generalized) Neumann-type expansion in terms of
the corresponding generalized static solution operator for low frequencies except for
some additional degenerate correction operators.
In \cite{complete} they discussed the case $N=3$ and
in \cite{linelae,linelaz} the case of odd space dimensions $N$\,.
For even dimensions $N$ some technical complications arise due to the appearance of logarithmic terms in
the Hankel function of integer order. So in even dimensions the results still hold true but the complexity of notations
and calculations increases considerably. For $N=2$ (the most complicated case!) Peter showed in \cite{peter} how to
do this for Helmholtz' equation.

So we may expect and will show in a forthcoming paper that a similar low frequency asymptotic holds true for
Maxwell equations in an exterior domain, i.e. for small frequencies $\omega$ and $J\in\nzn$ we will prove
$$\loesom-\omega^\me\Pi-\sum_{j=0}^{J-1}\omega^j\loes^{j+1}\Pi_{\reg}
-\sum_{j=0}^{J-N}\omega^{j+N-1}\Gamma_j=\calO\big(|\omega|^J\big)$$
with projections $\Pi$ and $\Pi_{\reg}$ onto irrotational and solenoidal fields as well as
some degenerate correction operators $\Gamma_j$ in the operator topology of weighted Sobolev spaces.

Motivated by these considerations and following Hermann Weyl \cite{weyl} we want to discuss in this paper
the generalized electro-magneto static Maxwell system
\begin{equation}
\rot E=G\qqtext{,}\pdiv\eps E=f\qqtext{,}\iota^* E=0\mylabel{staticMaxsystem}
\end{equation}
in an exterior domain $\Omega\subset\rN$ using alternating differential forms. Here the `electric field' $E$
is a differential form of rank $q$ ($q$-form) and the data $G$ and $f$ are $(q+1)$- resp. $(q-1)$-forms.
To invoke suggestively the applicational background of the electro-magneto statics, it has become customary to denote
the exterior derivative $\pd$\, by $\rot$ and the co-differential $\delta$ by $\pdiv$\,. Thus we have on $q$-forms
$$\pdiv=(-1)^{(q-1)N}*\rot*\qquad.$$
Here $*$ is the Hodge star-operator. Furthermore, $\eps$ is a linear transformation acting on $q$-forms,
$\iota:\p\Omega\hookrightarrow\ol{\Omega}$ the natural embedding and $\iota^*$ the pull-back map of $\iota$\,.
So $\iota^* E$ can be considered as the restriction of the form $E$ to the $(N-1)$-dimensional Riemannian submanifold $\p\om$\,, the boundary.

In classical terms, i.e. $N=3$\,, $q=1$\,, identifying $1$- and $2$-forms with vector fields
(via the Riesz representation theorem and the star-operator) and $0$- and $3$-forms with scalar functions,
the system \eqref{staticMaxsystem} reads
\begin{equation}
\curl E=G\qqtext{,}\pdiv\eps E=f\qqtext{,}\restr{\nu\times E}{\p\Omega}=0\qquad,\mylabel{Estatic}
\end{equation}
where $\curl=\nabla\,\times$ resp. $\pdiv=\nabla\,\cdot$ is the classical rotation resp. divergence
and $\nu$ the outward unit normal at the boundary $\p\Omega$\,. Here we denote by $\nabla$ the classical gradient
and by $\times$ the vector-product in $\rd$\,. So in this case we get the classical electro static system
for the electric field with prescribed tangential component at the boundary.
By the vanishing tangential component of $E$ at the boundary we model total reflection of the electric field at the oundary,
i.e. the complement $\rd\ohne\Omega$ is a perfect conductor. Now physically $f$ is the charge density,
$\eps$ the dielectricity of the medium $\Omega$\,, $\eps E$ the displacement current and $G=0$\,.

Setting $N=3$ and $q=2$ another classical case appears. In this case $H:=\eps E$ is the magnetic field,
$\mu:=\eps^\me$ the permeability of our medium, $\mu H$ the magnetic induction, $f$ the current and $G=0$\,.
Now the system \eqref{staticMaxsystem}, i.e.
$$\rot\mu H=G\qqtext{,}\pdiv H=f\qqtext{,}\iota^*\mu H=0\qquad,$$
turns into (using the classical language)
\begin{equation}
\pdiv\mu H=G\qqtext{,}-\curl H=f\qqtext{,}\restr{\nu\cdot\mu H}{\p\Omega}=0\qquad.\mylabel{Hstatic}
\end{equation}
Thus we obtain the classical magneto static system of a perfect conductor corresponding to \eqref{Estatic}.

We will show in this paper that the static Maxwell problem \eqref{staticMaxsystem} has a solution for data
taken from a closed subspace of $\qLzom{q-1}{s}\times\Lzqpeom{s}$ with weights $s>1-N/2$\,.
Here for $s\in\rz$ we denote by $\Lzqsom$ the Hilbert space of all measurable $q$-forms $E$\,,
for which $\rho^s E$ is square integrable over $\om$\,, i.e.
$$\skp{E}{E}_{\Lzqsom}:=\intom\rho^{2s}E\wedge*\bar{E}<\infty\qquad.$$
(Here $\wedge$ is the exterior product and $\,\bar{\cdot}\,$ denotes complex conjugation.)

Because of the existence of a non-trivial $\lz$-kernel of \eqref{staticMaxsystem},
the harmonic Dirichlet forms $\dhqepsom{}$\,, our solution is unique, if we impose some adequate
orthogonality constraints on it. We receive solutions, which lie in the naturally expected weighted
Sobolev space $\Lzqom{s-1}$ except of a finite sum of special generalized spherical harmonics.
To obtain our results we consider linear, bounded, symmetric and uniformly positive definite transformations $\eps$\,,
such that the perturbations $\epsd:=\eps-\id$ are $\pc{1}$
in the outside of an arbitrary compact set and decay with order $\tau>0$\,, i.e.
$$\p^\alpha\epsd=\calO(r^{-|\alpha|-\tau})\qqtext{as}r\to\infty\qqtext{for all}|\alpha|\leq1\qquad.$$
Depending on the weight $s$ we have to adjust the order of decay monotone increasing.

A solution theory for the static system \eqref{staticMaxsystem} has been given
by Kress \cite{kress} and Picard \cite{potential} for homogeneous, isotropic media, i.e. $\eps=\id$\,, and by Picard
\cite{decomposition} for inhomogeneous, anisotropic media. (Here $\eps$ is even allowed to be a
non-linear transformation as mentioned above.) Moreover, in the classical cases of electro- and magneto-statics
Picard \cite{boundaryelectro} and Milani and Picard \cite{milani} developed a solution theory
for inhomogeneous, anisotropic media. In these papers the data are taken from some
closed subspaces of $\lz$ or $\Lz{1}$\,. This means in our notation that until now solution theories
for the special cases of weights $s=0$ and $s=1$ are known.

Keeping in mind that we eventually want to be able to define a generalized Neumann-type expansion,
our immediate goal is to construct a special static solution operator $\loes$
associated with the electro-magneto static Maxwell system
\begin{align*}
\rot E&=G&&,&\pdiv\eps E&=0&&,&\iota^* E=0&&,\\
\pdiv H&=F&&,&\rot\mu H&=0&&,&\iota^*\mu H=0&&,
\end{align*}
which maps data $\FG$ from the closed subspace
$$\big(\dqsnom\cap\dhqom^\bot\big)\times\big(\ronqpesnom\cap\dhqpeom^\bot\big)$$
of $\Lzqsom\times\Lzqpesom$ to the forms $(\eps E,\mu H)$ from
$$\big(\eps\ronqtom\times\mu\dqpetom\big)
\cap\Big(\big(\dqtnom\cap\dhqom^\bot\big)\times\big(\ronqpetnom\cap\dhqpeom^\bot\big)\Big)$$
with some $t\leq s-1$ and $t<N/2$\,. The main tool for this is the construction of
`towers' of special homogeneous (iterated) static solutions in the whole space.
Here for $t\in\rz$ we introduce the weighted Sobolev spaces suited for Maxwell equations
\begin{align*}
\rqtom&:=\setb{E\in\Lzqtom}{\rot E\in\Lzqpeom{t+1}}\qquad,\\
\dqpetom&:=\setb{H\in\Lzqpetom}{\pdiv H\in\Lzqom{t+1}}
\end{align*}
equipped with their canonical norms. Furthermore, we generalize the homogeneous boundary condition in $\ronqtom$\,,
the closure of $\cqun(\Omega)$ in the $\rqtom$-norm. The symbol $\bot$ stands for orthogonality
with respect to the $\lz$-scalar product.
A subscript $0$ at the lower left corner
indicates vanishing rotation resp. divergence. If $\Omega=\rN$ we omit the dependence on the domain.
All these spaces are Hilbert spaces.

In a final section we deal with inhomogeneous boundary conditions. Using a new result from Weck \cite{wecklip},
which allows to define traces of $q$-forms on domains with Lipschitz-boundaries, we discuss the static problem
$$\rot E=G\qqtext{,}\pdiv\eps E=f\qqtext{,}\iota^*E=\lambda\qquad.$$
It turns out that the solution theory for this problem is an easy consequence of the results for homogeneous boundary conditions
and the existence of an adequate extension operator for our traces.

In this paper we follow closely the ideas of \cite{complete} and \cite{linelaz}.
We note that dual results can easily be obtained utilizing the Hodge star-operator, but for sake of brevity
we shall refrain from stating those results explicitly. Moreover, to decrease the complexity of this paper we only deal
with \ul{odd} space dimensions $N\geq3$ to avoid logarithmic terms as described above. We mention that many results hold for even
dimensions $N$ as well. Essentially we need the assumption $N$ odd only to construct our towers and these again are used to iterate our
solution operators. So the static solution theory still remains valid, if $N$ is even.

Throughout this paper we use the notations from \cite{paulytimeharm} resp. \cite{paulydiss} as mentioned above.
(Especially in this paper the index `loc' assigned to spaces is always to be understood
in the sense of $\omq$\,. Moreover, the index `vox' denotes compact supports.)

Essentially the present paper is the second part of the author's doctoral thesis, \cite{paulydiss}.
For sake of brevity, however, some proofs are merely sketched or completely omitted.
For more details and some additional results the interested reader is referred to \cite{paulydiss}.

This paper is the second one in a series of three papers having the aim to determine the low
frequency asymptotic of the time-harmonic Maxwell equations completely.
In the first paper \cite{paulytimeharm} we discussed the time-harmonic solution operator and showed its convergence
to a static solution operator as the frequency tends to zero.
With the present paper we provide the means to define higher powers of the static solution operators in suitably
weighted Sobolev-type spaces. In the forthcoming final paper of the series we shall then develop
a generalized Neumann-type expansion to fully analyze the low-frequency behavior of time-harmonic solutions
of Maxwell's equations.

\section{Towers of static solutions in the whole space}

In this section we consider the homogeneous and isotropic whole space case, i.e.
$$\Omega=\rN\qqtext{,}\eps=\id\qquad,$$
and assume $3\leq N\in\nz$ to be \ul{odd}.\mylabel{sectiontower}
Our aim is to provide a generalized spherical harmonics expansion for differential forms.
We use the spherical calculus (and its notations) developed by Weck and Witsch in \cite{sphharm}
(a technique to use polar-coordinates for $q$-forms) to construct towers of homogeneous differential forms
$$\dqksmpm\qqtext{and}\rqksmpm\qqtext{in}\cqu\big(\rN\ohne\{0\}\big)$$
solving the following system for $k\in\nzn$:
\begin{align}
\rot\dqnsmpm&=0&&,&\pdiv\rqpensmpm&=0\mylabel{Turmeins}\\
\pdiv\dqksmpm&=0&&,&\rot\rqpeksmpm&=0\mylabel{Turmzwei}\\
\rot\dqksmpm&=\rqpekmesmpm&&,&\pdiv\rqpeksmpm&=\dqkmesmpm\mylabel{Turmdrei}
\end{align}
These towers coincide in some sense with the eigenforms
$S^q_{\sigma,m}$\,, $T^q_{\sigma,m}$ of the Laplace-Beltrami operator on the unit sphere $\SN\subset\rN$\,,
which establish a complete orthonormal system in $\lzq(\SN)$ and solve the Maxwell eigenvalue system
$$\Rot T^q_{\sigma,m}=\ie\omega^q_\sigma\cdot S^{q+1}_{\sigma,m}\qqtext{,}\Div S^{q+1}_{\sigma,m}=\ie\omega^q_\sigma\cdot T^q_{\sigma,m}\qquad,$$
where $\omega^{q-1}_\sigma:=(q+\sigma)^\peh\cdot(q'+\sigma)^\peh$ with $q':=N-q$\,.
Here $\Rot$ and $\Div$ denote the exterior derivative and co-derivative on the unit sphere.

For the construction of these towers we use the operators $\rho$\,, $\tau$ and their right inverses
$\rhoh$\,, $\tauh$ introduced in \cite{sphharm}, intensively.
These towers have already been defined and discussed in \cite[p. 1503]{linelaz}.
For our purposes we have to study them more thoroughly. For $k,\sigma\in\nzn$ we let
$$\homg{k}{\sigma}{\pm}:=\begin{cases}k+\sigma&,\quad\pm=+\\k-\sigma-N&,\quad\pm=-\end{cases}$$
and with $\mu^q_\sigma:=\mu^{q,N}_\sigma$ from \cite[p. 1029, Theorem 1 (iii)]{sphharm} we introduce

\begin{defini}\mylabel{turmdef}
Let $q\in\{0,\dots,N\}$\,, $k,\sigma\in\nzn$ and $m\in\{1,\dots,\mu^q_\sigma\}$\,. Then we define `tower-forms' by
\begin{align*}
\turmd{q}{2k}{\sigma}{m}{\pm}&:={}^\pm\alpha^{q,k}_\sigma\cdot r^{\homg{2k}{\sigma}{\pm}}\cdot\big(-\ie\omega^{q-1}_{\sigma}\rhoh T^{q-1}_{\sigma,m}+(q'+\homg{2k}{\sigma}{\pm})\,\tauh S^q_{\sigma,m}\big)\qquad,\\
\turmd{q-1}{2k+1}{\sigma}{m}{\pm}&:={}^\pm\alpha^{q,k}_\sigma\cdot r^{\homg{2k+1}{\sigma}{\pm}}\cdot\tauh T^{q-1}_{\sigma,m}\qquad,\\
\turmr{q}{2k}{\sigma}{m}{\pm}&:={}^\pm\alpha^{q,k}_\sigma\cdot r^{\homg{2k}{\sigma}{\pm}}\cdot\big((q+\homg{2k}{\sigma}{\pm})\,\rhoh T^{q-1}_{\sigma,m}+\ie\omega^{q-1}_{\sigma}\tauh S^q_{\sigma,m}\big)\qquad,\\
\turmr{q+1}{2k+1}{\sigma}{m}{\pm}&:={}^\pm\alpha^{q,k}_\sigma\cdot r^{\homg{2k+1}{\sigma}{\pm}}\cdot\rhoh S^q_{\sigma,m}\qquad.
\end{align*}
The coefficients satisfy the recursion
$${}^\pm\alpha^{q,k}_\sigma:=\frac{{}^\pm\alpha^{q,k-1}_\sigma}{2k\cdot(2k\pm2\sigma\pm N)}\qtext{,}{}^-\alpha^{q,0}_\sigma:=1\qtext{,}{}^+\alpha^{q,0}_\sigma:=\frac{(-1)^{1+\delta_{q,0}+\delta_{q,N}}}{2\sigma+N}\quad.$$
Moreover, we collect all these indices in an index $I:=(\sgn,k,\sigma,m)$ taken from the set
$\{\pm\}\times\nzn\times\nzn\times\nz$ and define the notation
$$D^q_I:=\turmd{q}{k}{\sigma}{m}{\sgn}\qqtext{and}R^q_I:=\turmr{q}{k}{\sigma}{m}{\sgn}\qquad.$$
Here we call $\s(I):=\sgn=\pm$ the `sign', $\he(I):=k\in\nzn$ the `height', $\ei(I):=\sigma\in\nzn$ the `eigenvalue index'
and $\ci(I):=m\in\nz$ the `counting index' of a tower-q-form $D^q_I$ or $R^q_I$\,.
Furthermore, we define the `homogeneity degree' of a tower-form by
$$\homd(D^q_I):=\homd(R^q_I):=\homg{}{I}{}:=\homg{\he(I)}{\ei(I)}{\s(I)}\qquad.$$
Finally we define the upper bound of the counting index
$$\mu^{q,k}_{\sigma}:=\begin{cases}\mu^q_\sigma&,\,k\text{ even}\\\mu^{q+1}_\sigma&,\,k\text{ odd}\end{cases}$$
and the two index sets
\begin{align*}
\cIq&:=\setb{I}{\s(I)\in\{+,-\}\,\wedge\,\he(I),\ei(I)\in\nzn\,\wedge\,1\leq\ci(I)\leq\mu^{q,\he(I)}_{\ei(I)}}\quad,\\
\cJq&:=\setb{J}{\s(J)\in\{+,-\}\,\wedge\,\he(J),\ei(J)\in\nzn\,\wedge\,1\leq\ci(J)\leq\mu^{q-1,\he(J)+1}_{\ei(J)}}\quad.
\end{align*}
\end{defini}

\begin{rem}\mylabel{TurmBemNull}
\begin{itemize}
\item[\rm\bf (i)] The recursion of the coefficients is well defined because $N$ is odd. Thus our tower-forms
are well defined. For even dimensions the recursion is also well defined for tower-forms with
positive sign and for tower-forms with negative sign as long as $k<N/2$\,. (In the last case we would
have to work with logarithmic terms of the radius $r$ for higher $k$\,.)
Therefore for even dimensions $N\geq4$ all tower-forms with negative sign up to heights three are well defined.
\item[\rm\bf (ii)] The tower-forms $D^q_I$ and $R^q_I$ are elements of $\cqu\big(\rN\ohne\{0\}\big)$\,,
homogeneous of degree $\homg{}{I}{}$ and solutions of the system \eqref{Turmeins}-\eqref{Turmdrei} in $\rN\ohne\{0\}$\,.
\item[\rm\bf (iii)] An index $I$ resp. $J$ of a tower-form $D^q_I$ resp. $R^q_J$ belongs to the index set $\cIq$ resp. $\cJq$\,.
\item[\rm\bf (iv)] The elements of the countable set of tower-forms
\begin{align*}
&\setb{D^q_I,R^q_J}{I\in\cIq,J\in\cJq\,\wedge\,\he(J)\geq1}\\
=\,&\setb{D^q_I,R^q_J}{I\in\cIq,J\in\cJq\,\wedge\,\he(I)\geq1}
\end{align*}
are linear independent.
\item[\rm\bf (v)] From the defining recursion of the coefficients we get the following explicit formulas:
\begin{align*}
{}^+\alpha^{q,k}_\sigma&=\frac{\Gamma(1+N/2+\sigma)}{4^k\cdot k!\cdot\Gamma(k+1+N/2+\sigma)}\cdot\frac{(-1)^{1+\delta_{q,0}+\delta_{q,N}}}{2\sigma+N}\\
{}^-\alpha^{q,k}_\sigma&=\frac{\Gamma(1-N/2-\sigma)}{4^k\cdot k!\cdot\Gamma(k+1-N/2-\sigma)}
\end{align*}
Here $\Gamma$ denotes the gamma-function.
The coefficients ${}^\pm\alpha^{q,k}_\sigma$ converge rapidly to zero as $k\to\infty$\,.
Thus for $0<a\leq b<\infty$ the tower-forms $D^q_I$ and $R^q_I$\,, $I=(\sgn,k,\sigma,m)$\,, together with all their
derivatives and even after multiplication with arbitrary powers of $r$ are uniformly bounded
with respect to $a\leq|x|\leq b$ and $k,\sigma,m\in\nzn$\,.
\item[\rm\bf (vi)] The definitions of the tower-forms in Definition \ref{turmdef} have to be understood
in the sense that all not defined terms are defined to be zero.
Thus only for the ranks $q\in\{1,\dots,N-1\}$ no problems
occur and we get `regular' tower-forms. In the extreme cases $q\in\{0,N\}$\,, where we have only
the index $(\sigma,m)=(0,1)$\,, the only tower-forms are
$\turmd{0}{2k}{0}{1}{\pm}$\,, $\turmr{N}{2k}{0}{1}{\pm}$\,, $\turmd{N-1}{2k+1}{0}{1}{\pm}$ and
$\turmr{1}{2k+1}{0}{1}{\pm}$\,. We note in this cases
\begin{align*}
\turmd{0}{0}{0}{1}{-}&=0&&,&\turmr{N}{0}{0}{1}{-}&=0&&,\\
\turmd{0}{0}{0}{1}{+}&\in\Lin\{\Eins\}&&,&\turmr{N}{0}{0}{1}{+}&\in\Lin\{*\Eins\}&&.
\end{align*}
\end{itemize}
\end{rem}

\begin{rem}\mylabel{TurmBemZwei}
Because of $\Delta=\rot\pdiv+\pdiv\rot$ (Here the Laplacian $\Delta$ acts on each Euclidean component of the differential form.)
all tower-forms $D^q_I,R^q_J$ of heights less or equal to one satisfy
$$\Delta D^q_I=\Delta R^q_J=0\qquad.$$
Therefore comparing these tower-forms with the potential forms discussed in \cite{sphharm} we obtain for $q\in\{1,\dots,N-1\}$
\begin{align*}
\turmd{q}{0}{\sigma}{m}{-}&=-(q+\sigma)^\peh(2\sigma+N)^\peh\cdot Q^{q,3}_{\sigma+2,m}&&,&\turmd{q-1}{1}{\sigma}{m}{-}&=Q^{q-1,2}_{\sigma+1,m}&&,\\
\turmr{q}{0}{\sigma}{m}{-}&=\ie(q'+\sigma)^\peh(2\sigma+N)^\peh\cdot Q^{q,3}_{\sigma+2,m}&&,&\turmr{q+1}{1}{\sigma}{m}{-}&=Q^{q+1,1}_{\sigma+1,m}&&,\\
\turmd{q}{0}{\sigma}{m}{+}&=\ie(q'+\sigma)^\peh(2\sigma+N)^\meh\cdot P^{q,4}_{\sigma,m}&&,&\turmd{q-1}{1}{\sigma}{m}{+}&=\frac{-1}{2\sigma+N}\cdot P^{q-1,2}_{\sigma+1,m}&&,\\
\turmr{q}{0}{\sigma}{m}{+}&=-(q+\sigma)^\peh(2\sigma+N)^\meh\cdot P^{q,4}_{\sigma,m}&&,&\turmr{q+1}{1}{\sigma}{m}{+}&=\frac{-1}{2\sigma+N}\cdot P^{q+1,1}_{\sigma+1,m}&&.
\end{align*}
In particular the tower-forms $\turmd{q}{0}{\sigma}{m}{-}$ and $\turmr{q}{0}{\sigma}{m}{-}$
resp. $\turmd{q}{0}{\sigma}{m}{+}$ and $\turmr{q}{0}{\sigma}{m}{+}$ are linear dependent, which we will
indicate by the symbol $\cong$\,. In detail we have
$$\turmd{q}{0}{\sigma}{m}{-}=\vartheta^q_\sigma\cdot\turmr{q}{0}{\sigma}{m}{-}\qqtext{,}
\turmr{q}{0}{\sigma}{m}{+}=\vartheta^q_\sigma\cdot\turmd{q}{0}{\sigma}{m}{+}\qquad,$$
where $\vartheta^q_\sigma:=\ie(q+\sigma)^{\frac{1}{2}}(q'+\sigma)^\meh$\,.
Furthermore, the potential forms $Q^{q,4}_{\sigma,m}$ resp. $P^{q,3}_{\sigma+2,m}$
are linear combinations of the tower-forms $\turmd{q}{2}{\sigma}{m}{-}$ and $\turmr{q}{2}{\sigma}{m}{-}$
resp. $\turmd{q}{2}{\sigma}{m}{+}$ and $\turmr{q}{2}{\sigma}{m}{+}$\,, i.e.
\begin{align*}
\frac{(2\sigma+N)^\peh}{2-2\sigma-N}Q^{q,4}_{\sigma,m}&=(q+\sigma)^\peh\turmr{q}{2}{\sigma}{m}{-}+\ie(q'+\sigma)^\peh\turmd{q}{2}{\sigma}{m}{-}&&,\\
\ie\frac{(2\sigma+N)^\meh}{2+2\sigma+N}P^{q,3}_{\sigma+2,m}&=(q'+\sigma)^\peh\turmr{q}{2}{\sigma}{m}{+}-\ie(q+\sigma)^\peh\turmd{q}{2}{\sigma}{m}{+}&&.
\end{align*}
For $q\in\{0,N\}$ we see
\begin{align*}
\turmd{0}{2}{0}{1}{-}&=\frac{-\ie}{2-N}\cdot Q^{0,4}_{0,1}&&,&\turmd{N-1}{1}{0}{1}{-}&=Q^{N-1,2}_{1,1}&&,\\
\turmr{N}{2}{0}{1}{-}&=\frac{1}{2-N}\cdot Q^{N,4}_{0,1}&&,&\turmr{1}{1}{0}{1}{-}&=Q^{1,1}_{1,1}&&,\\
\turmd{0}{0}{0}{1}{+}&=-\ie\cdot P^{0,4}_{0,1}&&,&\turmd{N-1}{1}{0}{1}{+}&=\frac{1}{N}\cdot P^{N-1,2}_{1,1}&&,\\
\turmr{N}{0}{0}{1}{+}&=P^{N,4}_{0,1}&&,&\turmr{1}{1}{0}{1}{+}&=\frac{1}{N}\cdot P^{1,1}_{1,1}&&.
\end{align*}
\end{rem}

The following picture explains the denotation `tower':
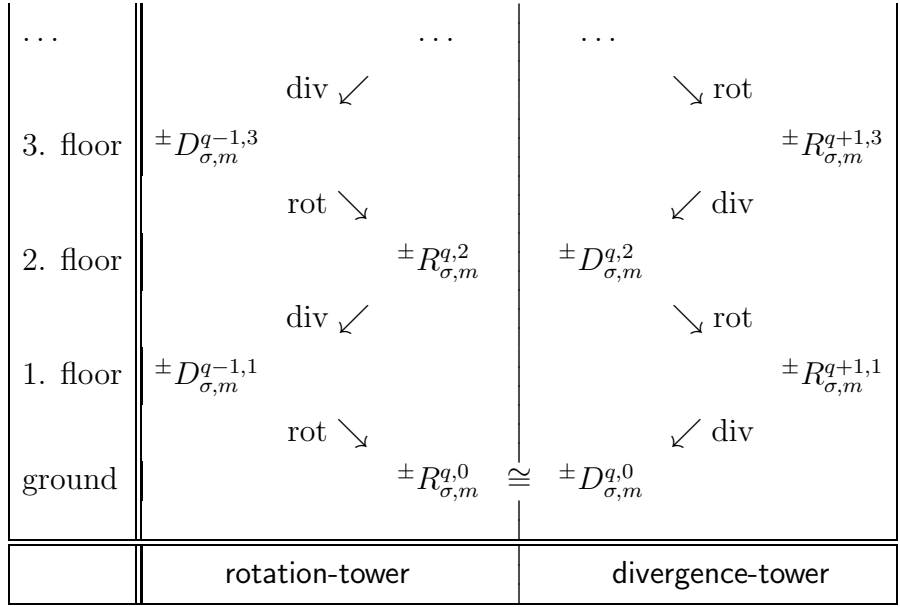
$$\begin{array}{|l||ccccccc|}
\dots&&&\dots&\Big|&\dots&&\\
&&\pdiv\swarrow&&\Big|&&\searrow\rot&\\
\text{\rm 3. floor}&{}^{\pm}D^{q-1,3}_{\sigma,m}&&&\Big|&&&{}^{\pm}R^{q+1,3}_{\sigma,m}\\
&&\rot\searrow&&\Big|&&\swarrow\pdiv&\\
\text{\rm 2. floor}&&&{}^{\pm}R^{q,2}_{\sigma,m}&\Big|&{}^{\pm}D^{q,2}_{\sigma,m}&&\\
&&\pdiv\swarrow&&\Big|&&\searrow\rot&\\
\text{\rm 1. floor}&{}^{\pm}D^{q-1,1}_{\sigma,m}&&&\Big|&&&{}^{\pm}R^{q+1,1}_{\sigma,m}\\
&&\rot\searrow&&\Big|&&\swarrow\pdiv&\\
\text{\rm ground}&&&{}^{\pm}R^{q,0}_{\sigma,m}&\cong&{}^{\pm}D^{q,0}_{\sigma,m}&&\\
&&&&\big|&&&\\
\hline\hline
&\multicolumn{3}{c}{\text{\sf rotation-tower}}&\Big|&\multicolumn{3}{c|}{\text{\sf divergence-tower}}\\
\hline
\end{array}$$

\begin{rem}\mylabel{TurmBemFuenf}
For an index $I\in\cIq$ resp. $J\in\cJq$ with \ul{odd} height we get $\rho D^q_I=0$ resp. $\tau R^q_J=0$
since $\rho\tauh=0$ resp. $\tau\rhoh=0$\,. Thus $TD^q_I=0$ resp. $RR^q_J=0$
with the operators $R=r\pd r\wedge\,=x_n\pd x^n\wedge\,$ and $T=\pm*R*$ from \cite{sphharm}.
Because of \eqref{Turmzwei} and with the commutator formula $C_{\pdiv,\varphi(r)}=\varphi'(r)r^\me T$ resp. $C_{\rot,\varphi(r)}=\varphi'(r)r^\me R$
\big(see e.g. \cite{sphharm} or {\paulytimeharmcomrotdiv}\big) we obtain
$$\pdiv\big(\varphi(r)D^q_I\big)=0\qqtext{resp.}\rot\big(\varphi(r)R^q_J\big)=0$$
for any $\varphi\in\pc{1}(\rz)$\,.
\end{rem}

To shorten the formulas we write for $\cI\subset\cIq$ resp. $\cJ\subset\cJq$
$$\calD^q(\cI):=\Lin\set{D^q_I}{I\in\cI}\qqtext{resp.}\calR^q(\cJ):=\Lin\set{R^q_J}{J\in\cJ}$$
\big(with the convention $\calD^q(\emptyset):=\{0\}$ resp. $\calR^q(\emptyset):=\{0\}$\big). For $s\in\rz$ let
$$\cI_s:=\setb{I\in\cI}{D^q_I\notin\Lzqs(A_1)}\qtext{,}
\cJ_s:=\setb{J\in\cJ}{R^q_J\notin\Lzqs(A_1)}\quad.$$
Furthermore, for $s\in\rz$ and $k,K\in\nzn$ we present the index sets
\begin{align*}
\pcIq{k}{}&:=\setb{I\in\cIq}{\he(I)=k}&&,&\pcIq{\leq K}{}&:=\bigcup_{k=0}^K\pcIq{k}{}&&,\\
\pcIq{k}{s}&:=\big(\pcIq{k}{}\big)_s&&,&\pcIq{\leq K}{s}&:=\bigcup_{k=0}^K\pcIq{k}{s}&&,\\
\pcIbq{k}{}&:=\setb{I\in\pcIq{k}{}}{\s(I)=-}&&,&\pcIbq{\leq K}{}&:=\bigcup_{k=0}^K\pcIbq{k}{}&&,\\
\pcIbq{k}{s}&:=\big(\pcIbq{k}{}\big)_s&&,&\pcIbq{\leq K}{s}&:=\bigcup_{k=0}^K\pcIbq{k}{s}
\end{align*}
and replacing $\cI$ by $\cJ$ similar index sets for $\cJ$\,.
Moreover, we introduce for indices
$$I:=(\sgn,k,\sigma,m)\in\cIq\qtext{resp.}J:=(\sgn,k,\sigma,m)\in\cJqpe$$
the negative indices
$$-I:=(-\sgn,k,\sigma,m)\in\cIq\qtext{resp.}-J:=(-\sgn,k,\sigma,m)\in\cJqpe$$
and with $j\in\zz$ for the shifted indices the notation
\begin{align*}
{}_jI&:=(\sgn,k+j,\sigma,m)\in\begin{cases}\cJqpe&\,,\,j\text{ odd}\\\cIq&\,,\,j\text{ even}\end{cases}
\intertext{resp.}
{}_jJ&:=(\sgn,k+j,\sigma,m)\in\begin{cases}\cIq&\,,\,j\text{ odd}\\\cJqpe&\,,\,j\text{ even}\end{cases}\qquad.
\intertext{For subsets $\cI$ resp. $\cJ$ of $\cIq$ resp. $\cJqpe$ we set}
\cIj&:=\set{{}_jI}{I\in\cI}\subset\begin{cases}\cJqpe&\,,\,j\text{ odd}\\\cIq&\,,\,j\text{ even}\end{cases}
\intertext{resp.}
\cJj&:=\set{{}_jJ}{J\in\cJ}\subset\begin{cases}\cIq&\,,\,j\text{ odd}\\\cJqpe&\,,\,j\text{ even}\end{cases}\qquad.
\end{align*}
With these definitions we then have
\begin{align*}
\rot D^q_{{}_1J}&=R^{q+1}_J&&,&\pdiv R^{q+1}_{{}_1I}&=D^q_I&&,\\
\rot D^q_I&=R^{q+1}_{{}_{-1}I}&&,&\pdiv R^{q+1}_J&=D^q_{{}_{-1}J}&&.
\end{align*}

\begin{rem}\mylabel{intbarkeitTuerme}
Let $m\in\nzn$ and $I\in\cIq$\,. Since our tower-forms are smooth and homogeneous we have for $s\in\rz$
\begin{align*}
&&D^q_I&\in\Lzqs(A_1)&&\equi&D^q_I&\in\hqms(A_1)\\
&\equi&\homg{}{I}{}&<-s-N/2&&\equi&s&<-\s(I)\big(\ei(I)+N/2\big)-\he(I)\qquad.
\end{align*}
If in particular $I\in\pcIbq{k}{}$\,, then $D^q_I$ is an element of $\hqms(A_1)$\,, if and only if
$$\ei(I)>s+k-N/2\qquad.$$
Thus for $k\in\nzn$ we can characterize our special index sets by
\begin{align*}
\pcIbq{k}{s}&=\setb{I\in\pcIbq{k}{}}{\ei(I)\leq s+k-N/2}\qquad,\\
\pcIbq{\leq k}{s}&=\setb{I\in\pcIbq{\leq k}{}}{\ei(I)\leq s+\he(I)-N/2}\qquad.
\end{align*}
We note that $\calD^q(\pcIbq{k}{s})=\calD^q(\pcIbq{\leq k}{s})=\{0\}$\,, if and only if $s<N/2-k$\,.
Thus for $s\geq N/2-k$ the spaces $\calD^q(\pcIbq{k}{s})$ and $\calD^q(\pcIbq{\leq k}{s})$ are subspaces of
$\qh{m}{q}{<\Nh-k}{}(A_1)$ but by definition even not of $\Lzqs(A_1)$\,.
Clearly all these assertions also hold true for tower-forms $R^q_J$ with $J\in\cJq$\,.
\end{rem}

Let us introduce the `matrix'-differential operator
\begin{equation}M:=\zmat{0}{\pdiv}{\rot}{0}\mylabel{Mdef}\end{equation}
acting on pairs of $(q,q+1)$-forms $\EH$ by
$$M\EH:=(\pdiv H,\rot E)\qquad.$$
Now we are able to prove the main result of this section, a generalized spherical harmonics expansion
suited for Maxwell equations. To this end we have to define for $K\in\nz$ and $s\in\rz$ some `exceptional' forms:
\begin{align}
\hat{D}^{q,K}&:=
\begin{cases}
\vspace{1mm}
\turmd{0}{K}{0}{1}{-}&\,,\,q=0\,\wedge\,K\text{ even}\\
\vspace{1mm}
\turmr{1}{1}{0}{1}{-}&\,,\,q=1\\
\vspace{1mm}
\turmd{N-1}{K}{0}{1}{-}&\,,\,q=N-1\,\wedge\,K\text{ odd}\\
0&\,,\,\text{otherwise}
\end{cases}\mylabel{ausnahmeformeins}\\
&\non\\
\hat{R}^{q+1,K}&:=
\begin{cases}
\vspace{1mm}
\turmr{1}{K}{0}{1}{-}&\,,\,q=0\,\wedge\,K\text{ odd}\\
\vspace{1mm}
\turmd{N-1}{1}{0}{1}{-}&\,,\,q=N-2\\
\vspace{1mm}
\turmr{N}{K}{0}{1}{-}&\,,\,q=N-1\,\wedge\,K\text{ even}\\
0&\,,\,\text{otherwise}
\end{cases}\mylabel{ausnahmeformzwei}\\
&\non\\
\hat{D}^{q,K}_s&:=
\begin{cases}
\vspace{1mm}
\turmd{0}{K}{0}{1}{-}&\,,\,q=0\,\wedge\,K\text{ even }\wedge\,s<N/2-K\\
\vspace{1mm}
\turmr{1}{1}{0}{1}{-}&\,,\,q=1\,\wedge\,s<N/2-1\\
\vspace{1mm}
\turmd{N-1}{K}{0}{1}{-}&\,,\,q=N-1\,\wedge\,K\text{ odd }\wedge\,s<N/2-K\\
0&\,,\,\text{otherwise}
\end{cases}\mylabel{ausnahmeformdrei}\\
&\non\\
\hat{R}^{q+1,K}_s&:=
\begin{cases}
\vspace{1mm}
\turmr{1}{K}{0}{1}{-}&\,,\,q=0\,\wedge\,K\text{ odd }\wedge\,s<N/2-K\\
\vspace{1mm}
\turmd{N-1}{1}{0}{1}{-}&\,,\,q=N-2\,\wedge\,s<N/2-1\\
\vspace{1mm}
\turmr{N}{K}{0}{1}{-}&\,,\,q=N-1\,\wedge\,K\text{ even }\wedge\,s<N/2-K\\
0&\,,\,\text{otherwise}
\end{cases}\mylabel{ausnahmeformvier}
\end{align}

\begin{theo}\mylabel{entwicklungssatzganzraum}
With $K\in\nz$ and $0\leq\tilde{r}<\bar{r}\leq\infty$ let $\EH$ denote a solution of the `iterated' Maxwell system
$$M^K\EH=(0,0)\qqtext{and}\pdiv E=0\qqtext{,}\rot H=0$$
in $Z_{\tilde{r},\bar{r}}$\,. Then $\EH\in\cqu(Z_{\tilde{r},\bar{r}})\times\cqpeu(Z_{\tilde{r},\bar{r}})$
and in $Z_{\tilde{r},\bar{r}}$ the representations
\begin{align}
\hspace{1cm}E&=\sum_{\substack{I\in\pcIq{\leq K-1}{}}}&e^{q,K}_I&\cdot D^q_I&&+&\hat{e}^{q,K}&\cdot\hat{D}^{q,K}\qquad,\hspace{1cm}\mylabel{entwe}\\
H&=\sum_{\substack{J\in\pcJqpe{\leq K-1}{}}}&h^{q+1,K}_J&\cdot R^{q+1}_J&&+&\hat{h}^{q+1,K}&\cdot\hat{R}^{q+1,K}\mylabel{entwz}
\end{align}
hold with unique constants $e^{q,K}_{\,\cdot\,}$\,, $h^{q+1,K}_{\,\cdot\,}\in\cz$ and $\hat{e}^{q,K}$\,, $\hat{h}^{q+1,K}\in\cz$\,,
provided that the exceptional forms do not vanish.

These series converge in $\cu(Z_{\tilde{r},\bar{r}})$\,,
i.e. uniformly together with all their derivatives in compact subsets of $Z_{\tilde{r},\bar{r}}$\,.

In the case $0<\tilde{r}<\bar{r}=\infty$ we have with some $s\in\rz$
$$\EH\in\hqms(A_{\tilde{r}})\times\hqpems(A_{\tilde{r}})\qqtext{for all}m\in\nzn\qquad,$$
if and only if all coefficients $e^{q,K}_I$ and $h^{q+1,K}_J$ with $\homg{}{I}{},\homg{}{J}{}\geq-s-N/2$ vanish.
This holds true, if and only if $\he(I)+\ei(I),\he(J)+\ei(J)\geq-s-N/2$ for indices $I,J$ with positive sign and
$\he(I)-\ei(I),\he(J)-\ei(J)\geq-s+N/2$ for indices $I,J$ with negative sign.
Then the series converge for all $\hat{r}>\tilde{r}$ uniformly together with all derivatives
even after multiplication with arbitrary powers of $r$ in $A_{\hat{r}}$\,.
Thus in particular they converge in $\hms(A_{\hat{r}})$\,.

Especially for $s\geq-N/2$ there appear only tower-forms with negative sign. In this case \eqref{entwe} and \eqref{entwz} turn to
\begin{align*}
\hspace{1cm}E&=\sum_{\substack{I\in\pcIbq{\leq K-1}{}\ohne\pcIbq{\leq K-1}{s}}}&e^{q,K}_I&\cdot D^q_I&&+&\hat{e}^{q,K}&\cdot\hat{D}^{q,K}_s&&,\\
H&=\sum_{\substack{J\in\pcJbqpe{\leq K-1}{}\ohne\pcJbqpe{\leq K-1}{s}}}&h^{q+1,K}_J&\cdot R^{q+1}_J&&+&\hat{h}^{q+1,K}&\cdot\hat{R}^{q+1,K}_s&&.\hspace{1cm}
\end{align*}
\end{theo}

\begin{proof}
The smoothness of $\EH$ follows by the regularity result {\paulydissregtheoaussen}.
Remark \ref{intbarkeitTuerme} yields the integrability properties of each single term in the stated expansion.
Concerning the mode of convergence we refer to \cite[p. 1033]{sphharm} and \cite[p. 1508, Theorem 1]{linelaz},
where similar expansions have been discussed. In particular for $\bar{r}=\infty$ the series converge in $\Lzs(A_{\hat{r}})$\,,
if and only if all terms in the expansion belong to $\Lzs(A_{\hat{r}})$\,.
Thus we only have to show the representation formulas \eqref{entwe} and \eqref{entwz}.

Let us look at $E$ in the case $K=1$\,. We have $\rot E=0$ and $\pdiv E=0$\,. Thus $E$ is a potential form, i.e. $\Delta E=0$\,,
and we obtain from \cite[p. 1033]{sphharm} the representation
\beq E=\sum_{k,\sigma,m}\alpha^q_{k,\sigma,m}\cdot P^{q,k}_{\sigma,m}+\sum_{k,\sigma,m}\beta^q_{k,\sigma,m}\cdot Q^{q,k}_{\sigma,m}\qtext{,}\alpha^q_{k,\sigma,m},\beta^q_{k,\sigma,m}\in\cz\quad.\mylabel{EentwicklungEins}\eeq
By testing the equation $\rot E=0$ with $\varphi(r)\rhoh T^q_{\sigma-1,m}$
for any $\varphi\in\cun\big((\tilde{r},\bar{r})\big)$\,, i.e. computing
$$0=\skpb{\rot E}{\varphi(r)\rhoh T^q_{\sigma-1,m}}_{\Lzqpe{}}$$
with partial integration and \eqref{EentwicklungEins}, we see
$\alpha^q_{2,\sigma,m}=\beta^q_{2,\sigma,m}=0$ except of $\beta^{N-1}_{2,1,1}$\,.
Testing with $\varphi(r)\rhoh S^q_{\sigma-2,m}$ yields $\alpha^q_{3,\sigma,m}=\beta^q_{4,\sigma-2,m}=0$\,.
Analogously we obtain from the equation $\pdiv E=0$ by testing with $\varphi(r)\rhoh T^{q-2}_{\sigma-1,m}$\, that
$\alpha^q_{1,\sigma,m}=\beta^q_{1,\sigma,m}$ must vanish except of $\beta^{1}_{1,1,1}$\,.
Finally only
$$E=\sum_{\sigma,m}\alpha^q_{4,\sigma,m}\cdot P^{q,4}_{\sigma,m}+\sum_{\sigma,m}\beta^q_{3,\sigma,m}\cdot Q^{q,3}_{\sigma,m}
+\begin{cases}
\vspace{1mm}
\beta^{1}_{1,1,1}\cdot Q^{1,1}_{1,1}&\,,\,q=1\\
\vspace{1mm}
\beta^{N-1}_{2,1,1}\cdot Q^{N-1,2}_{1,1}&\,,\,q=N-1\\
0&\,,\,\text{otherwise}
\end{cases}$$
remains from \eqref{EentwicklungEins}.
With Remark \ref{TurmBemZwei} we can replace these potential forms by our tower-forms and we receive the asserted representation.
Because $H$ solves the same system as $E$ (replacing $q$ by $q+1$) we obtain the representation for $H$ in the case $K=1$ as well.

Assuming now that our representations hold for some $K\geq1$\,,
we consider $E$ solving the system $M^{K+1}(E,0)=(0,0)$ and $\pdiv E=0$\,. 
Then the form $H:=\rot E$ satisfies
$M^K(0,H)=(0,0)$ and $\rot H=0$\,. Our assumptions for $K$ yield
$$H=\sum_{\substack{J\in\pcJqpe{\leq K-1}{}}}h^{q+1,K}_J\cdot R^{q+1}_J+\hat{h}^{q+1,K}\cdot\hat{R}^{q+1,K}$$
and with the ansatz
$$\tilde{E}:=\sum_{\substack{J\in\pcJqpe{\leq K-1}{}}}h^{q+1,K}_J\cdot D^q_{{}_1J}
+\hat{h}^{q+1,K}\cdot
\begin{cases}
\vspace{1mm}
\turmd{0}{K+1}{0}{1}{-}&\,,\,q=0\,\wedge\,K+1\text{ even}\\
\vspace{1mm}
\turmd{N-1}{K+1}{0}{1}{-}&\,,\,q=N-1\,\wedge\,K+1\text{ odd}\\
0&\,,\,\text{otherwise}
\end{cases}$$
we see that $e:=E-\tilde{E}$ solves the system
$$\pdiv e=0\qqtext{,}\rot e=\hat{h}^{q+1,K}\cdot\begin{cases}\turmd{N-1}{1}{0}{1}{-}&\,,\,q=N-2\\0&\,,\,\text{otherwise}\end{cases}\qquad.$$
If $q\neq N-2$ the conclusion for $K=1$ gives
$$e=\sum_{\substack{I\in\pcIq{0}{}}}e^{q,1}_I\cdot D^q_I+\hat{e}^{q,1}\cdot\hat{D}^{q,1}$$
and thus
\begin{align*}
E=e+\tilde{E}&=\sum_{\substack{I\in\cIq\,,\\1\leq \he(I)\leq K}}h^{q+1,K}_{{}_{-1}I}\cdot D^q_I
+\sum_{\substack{I\in\pcIq{0}{}}}e^{q,1}_I\cdot D^q_I\\
&\\
&\qquad\qquad+
\begin{cases}
\vspace{1mm}
\hat{h}^{q+1,K}\cdot\turmd{0}{K+1}{0}{1}{-}&\,,\,q=0\,\wedge\,K+1\text{ even}\\
\vspace{1mm}
\hat{e}^{q,1}\cdot\turmr{1}{1}{0}{1}{-}&\,,\,q=1\\
\vspace{1mm}
\hat{e}^{q,1}\cdot\turmd{N-1}{1}{0}{1}{-}&\,,\,q=N-1\\
\vspace{1mm}
\hat{h}^{q+1,K}\cdot\turmd{N-1}{K+1}{0}{1}{-}&\,,\,q=N-1\,\wedge\,K+1\text{ odd}\\
0&\,,\,\text{otherwise}
\end{cases}\quad,
\end{align*}
which had to be shown. If $q=N-2$ we only have
$$\pdiv e=0\qqtext{and}\Delta e=0\qquad.$$
Following the arguments used in the case $K=1$ we obtain for $e$ the expansion \eqref{EentwicklungEins} and reduce this representation
similarly with the additional information $\pdiv e=0$ to
$$e=\sum_{\substack{k=2,4\,,\\\sigma,m}}\alpha^{N-2}_{k,\sigma,m}\cdot P^{N-2,k}_{\sigma,m}+\sum_{\substack{k=2,3\,,\\\sigma,m}}\beta^{N-2}_{k,\sigma,m}\cdot Q^{N-2,k}_{\sigma,m}
+\begin{cases}
\vspace{1mm}
\beta^{1}_{1,1,1}\cdot Q^{1,1}_{1,1}&\,,\,q=1\\
0&\,,\,\text{otherwise}
\end{cases}\quad.$$
Remark \ref{TurmBemZwei} yields with new constants
$$e=\sum_{\substack{I\in\pcIq{\leq1}{}}}e^{N-2,2}_I\cdot D^{N-2}_I
+\hat{e}^{1,2}\cdot\begin{cases}
\vspace{1mm}
\turmr{1}{1}{0}{1}{-}&\,,\,q=1\\
0&\,,\,\text{otherwise}
\end{cases}\qquad.$$
Analogously we show the representation for $H$\,. Our proof is complete.
\end{proof}

\section{Electro-magneto static operators}

From now on we want to discuss the inhomogeneous, anisotropic (generalized) static Maxwell equations in an exterior domain
$$\Omega\subset\rN\qqtext{,}3\leq N\in\nz\qquad,$$
and fix a radius $r_0>0$ and some radii $r_n:=2^nr_0$\,, $n\in\nz$\,, such that
$$\rN\ohne\Omega\subset U_{r_0}\qquad.$$
We assume $\Omega$ to possess the `Maxwell local compactness property' {\sf MLCP} \big(see {\paulytimeharmmlcpdef}\big), i.e. the inclusions
$$\Ronqom{}\cap\Dqom{}\hookrightarrow\Lzqloc(\omq)$$
are compact for all $q$\,. Furthermore, we remind of the cut-off functions
$\mbox{\boldmath$\eta$}$\,, $\hat{\eta}$ and $\eta$ from {\paulytimeharmetadef}.
Let $\eps=\id+\epsd$ and $\mu=\id+\mud$ be two $\tau$-$\pc{1}$-admissible \big(see \paulytimeharmdefeps\big)
transformations on $q$- and $(q+1)$-forms with some $\tau\geq0$\,, which will vary throughout this paper.
The greek letter $\tau$ always stands for the order of decay of the perturbations $\epsd$ and $\mud$\,.

Moreover, we introduce the Dirichlet forms defined in {\paulytimeharmepsdef}
$$\dhqepsom{t}=\ronqnom{t}\cap\eps^\me\dqnom{t}\qqtext{,}t\in\rz$$
and obtain our first result on the integrability of Dirichlet forms:\mylabel{statopweight}

\begin{lem}\mylabel{intdirichletid}
$\dH{q}{-\Nh}{}(\Omega)=\dhqom=\dH{q}{<\Nh-1}{}(\Omega)$ and even $\dhqom=\dH{q}{<\Nh}{}(\Omega)$ holds,
if $q\notin\{1,N-1\}$\,.
\end{lem}

\begin{rem}\mylabel{intdirichletidbem}
In particular if $s>1-N/2$ then $\dhqom$ is a closed subspace of $\Lzqom{-s}$\,, the dual space of $\Lzqsom$\,.
If $q\notin\{1,N-1\}$ this remains valid for $s>-N/2$\,.
\end{rem}

\begin{proof}
Applying Theorem \ref{entwicklungssatzganzraum} with $s=-N/2$ we have for $E\in\dH{q}{-\Nh}{}(\Omega)$
$$\restr{E}{A_{r_0}}=\sum_{I\in\pcIbq{0}{}}e^{q,1}_I\cdot D^q_I+\hat{e}^{q,1}\cdot\hat{D}^{q,1}_{-\Nh}$$
because of $\pcIbq{0}{s}=\emptyset$\,. By Remark \ref{intbarkeitTuerme} the sum is an element of $\Lzqom{<\Nh}$ and
the second term, which vanishes for $q\notin\{1,N-1\}$\,, lies in $\Lzqom{<\Nh-1}$\,.
\end{proof}

One easily concludes with \paulytimeharmpartint:

\begin{cor}\mylabel{dirichletsenkrecht}
Let $s>1-N/2$\,. Then with closures taken in $\Lzqsom$
$$\ol{\rot\pr{q-1}{s-1}{}{\circ}(\Omega)}\cup\ol{\rot\pR{q-1}{s}{}{\circ}(\Omega)}\cup\ol{\pdiv\pdi{q+1}{s-1}{}{}(\Omega)}\cup\ol{\pdiv\pDi{q+1}{s}{}{}(\Omega)}\subset\dhqom^\bot$$
holds true. Here we denote by $\bot$ the orthogonality in $\Lzqom{}$\,, i.e. in the sense of the
$\Lzqsom$-$\Lzqom{-s}$ duality.
\end{cor}

Let us now closely follow the constructions in \cite[p. 1631]{complete} or \cite[p. 1511]{linelaz}.
For some $t\in\rz$ we consider vector spaces of the form
$$U^q_t(\Omega):=V^q_t(\Omega)+\eta\calD^q(\cI)\qqtext{,}\cI\subset\cIq$$
with some Hilbert spaces $V^q_t(\Omega)\subset\Lzqtom$\,, e.g. $V^q_t(\Omega)=\ronqtom\cap\eps^\me\dqtom$\,.
Because of the smoothness and integrability properties of our tower-forms we have
$$U^q_t(\Omega)=V^q_t(\Omega)\dotplus\eta\calD^q(\cI_t)\qquad\qqtext{($\dotplus$\,: direct sum)}\qquad.$$
We define an inner product in $U^q_t(\Omega)$\,, such that
\begin{itemize}
\item\quad in $V^q_t(\Omega)$ the original scalar product is kept;
\item\quad $\eta\calD^q(\cI_t)=\set{\eta D^q_I}{I\in\cI_t}$ is an orthonormal system;
\item\quad the sum $V^q_t(\Omega)\dotplus\eta\calD^q(\cI_t)=V^q_t(\Omega)\boxplus\eta\calD^q(\cI_t)$ is orthogonal.
\end{itemize}
Clearly we extend these definitions replacing $\eta\calD^q(\cI)$ by $\eta\calR^q(\cJ)$\,.
$U^q_t(\Omega)$ is a Hilbert space, if and only if $\cI_t$ is finite, and
independent of the cut-off function $\eta$ in the following sense:
If $\xi$ is another cut-off function with the same properties, then the two Hilbert spaces
$V^q_t(\Omega)\dotplus\xi\calD^q(\cI_t)$ and $V^q_t(\Omega)\dotplus\eta\calD^q(\cI_t)$
have different scalar products but coincide as sets and the identity mapping is a topological isomorphism between them
with norm depending on $\eta$ and $\xi$\,.
\big(We note $E+\eta T=E+(\eta-\xi)T+\xi T$ and $\supp(\eta-\xi)$ is a compact subset of $\om$\,.\big)

We introduce the special forms
$$\check{D}^{q,K}_s\qqtext{and}\check{R}^{q+1,K}_s$$
replacing $<$ by $\geq$ in the definitions of $\hat{D}^{q,K}_s$ and $\hat{R}^{q+1,K}_s$
in \eqref{ausnahmeformdrei} and \eqref{ausnahmeformvier}. In particular we have for $K=1$
\begin{align*}
\check{D}^{q,1}_s&=\check{R}^{q,1}_s=
\begin{cases}
\vspace{1mm}
\turmr{1}{1}{0}{1}{-}&\,,\,q=1\,\wedge\,s\geq N/2-1\\
\vspace{1mm}
\turmd{N-1}{1}{0}{1}{-}&\,,\,q=N-1\,\wedge\,s\geq N/2-1\\
0&\,,\,\text{otherwise}
\end{cases}
\intertext{and for $K=2$}
\check{D}^{q,2}_s&=
\begin{cases}
\vspace{1mm}
\turmd{0}{2}{0}{1}{-}&\,,\,q=0\,\wedge\,s\geq N/2-2\\
\vspace{1mm}
\turmr{1}{1}{0}{1}{-}&\,,\,q=1\,\wedge\,s\geq N/2-1\\
0&\,,\,\text{otherwise}
\end{cases}\qquad,\\
&\\
\check{R}^{q+1,2}_s&=
\begin{cases}
\vspace{1mm}
\turmd{N-1}{1}{0}{1}{-}&\,,\,q=N-2\,\wedge\,s\geq N/2-1\\
\vspace{1mm}
\turmr{N}{2}{0}{1}{-}&\,,\,q=N-1\,\wedge\,s\geq N/2-2\\
0&\,,\,\text{otherwise}
\end{cases}\qquad.
\end{align*}
Moreover, we set $\check{\calD}^{q,K}_s:=\Lin\check{D}^{q,K}_s$ and $\check{\calR}^{q+1,K}_s:=\Lin\check{R}^{q+1,K}_s$\,.

We need one more technical lemma:

\begin{lem}\mylabel{deltaNullzerleg}
Let $\hat{r}\geq r_0$\,, $s\geq-N/2$ and $E\in\Lzqom{-\Nh}$ be a solution of $\Delta E=0$ in $A_{\hat{r}}$\,.
Then $E$ is represented in $A_{\hat{r}}$ by
$$E-\sum_{\substack{k,\sigma,m\,,\\\sigma>2+s-N/2}}\beta_{k,\sigma,m}\cdot Q^{q,k}_{\sigma,m}$$
$$=\begin{cases}
\quad\bds\sum_{J\in\pcJbq{\leq1}{s}}\eds h^{q,2}_J\cdot R^q_J+\check{h}^{q,2}\cdot\check{R}^{q,2}_s&\text{, if }\rot E\in\Lzqpe{s+1}(A_{\hat{r}})\\
\quad\bds\sum_{I\in\pcIbq{\leq1}{s}}\eds e^{q,2}_I\cdot D^q_I+\check{e}^{q,2}\cdot\check{D}^{q,2}_s&\text{, if }\pdiv E\in\qLz{q-1}{s+1}(A_{\hat{r}})\\
\quad\bds\sum_{I\in\pcIbq{0}{s}}\eds e^{q,1}_I\cdot D^q_I+\check{e}^{q,1}\cdot\check{D}^{q,1}_s&\text{, if }\rot E\in\Lzqpe{s+1}(A_{\hat{r}})\\
&\quad\text{and }\pdiv E\in\qLz{q-1}{s+1}(A_{\hat{r}})\\
\bds\sum_{\substack{k,\sigma,m\,,\\\sigma\leq2+s-N/2}}\eds \beta_{k,\sigma,m}\cdot Q^{q,k}_{\sigma,m}&\text{, otherwise}
\end{cases}$$
with unique constants $\beta_{k,\sigma,m},e^{q,1}_I,e^{q,2}_I,h^{q,2}_J\in\cz$ and
$\check{e}^{q,1},\check{e}^{q,2},\check{h}^{q,2}\in\cz$\,, provided that the exceptional forms do not vanish.
Thus
$$E\in\Lzqsom\boxplus\begin{cases}
\quad\eta\calR^q(\pcJbq{\leq1}{s})\boxplus\eta\check{\calR}^{q,2}_s&\text{, if }\rot E\in\Lzqpe{s+1}(A_{\hat{r}})\\
\quad\eta\calD^q(\pcIbq{\leq1}{s})\boxplus\eta\check{\calD}^{q,2}_s&\text{, if }\pdiv E\in\qLz{q-1}{s+1}(A_{\hat{r}})\\
\quad\eta\calD^q(\pcIbq{0}{s})\boxplus\eta\check{\calD}^{q,1}_s&\text{, if }\rot E\in\Lzqpe{s+1}(A_{\hat{r}})\\
&\quad\text{and }\pdiv E\in\qLz{q-1}{s+1}(A_{\hat{r}})
\end{cases}$$
holds. We note $\eta\calD^q(\pcIbq{0}{s})\boxplus\eta\check{\calD}^{q,1}_s=\eta\calR^q(\pcJbq{0}{s})\boxplus\eta\check{\calR}^{q,1}_s$\,.
\end{lem}

\begin{proof}
Similarly to \eqref{EentwicklungEins} we have
$$\restr{E}{A_{\hat{r}}}=\sum_{k,\sigma,m}\beta_{k,\sigma,m}\cdot Q^{q,k}_{\sigma,m}\in\Lzq{<\Nh-2}(A_{\hat{r}})$$
because $E\in\Lzqom{-\Nh}$ and therefore no potential forms $P^{q,k}_{\sigma,m}$ occur.
Thus only the case $s\geq N/2-2$ is interesting and may be assumed from now on. With the properties
of $Q^{q,k}_{\sigma,m}$ we obtain \big(except of the forms $Q^{0,4}_{0,1}$\,, $Q^{1,1}_{1,1}$\,, $Q^{N-1,2}_{1,1}$ and $Q^{N,4}_{0,1}$\big)
\begin{align*}
\restr{\rot E}{A_{\hat{r}}}&=\sum_{\sigma,m}\tilde{\beta}_{2,\sigma,m}\cdot Q^{q+1,3}_{\sigma+1,m}+\sum_{\sigma,m}\tilde{\beta}_{4,\sigma,m}\cdot Q^{q+1,1}_{\sigma+1,m}\qquad,\\
\restr{\pdiv E}{A_{\hat{r}}}&=\sum_{\sigma,m}\tilde{\tilde{\beta}}_{1,\sigma,m}\cdot Q^{q-1,3}_{\sigma+1,m}+\sum_{\sigma,m}\tilde{\tilde{\beta}}_{4,\sigma,m}\cdot Q^{q-1,2}_{\sigma+1,m}
\end{align*}
with new constants satisfying $\tilde{\beta}_{k,\sigma,m}=0\equi\beta_{k,\sigma,m}=0\equi\tilde{\tilde{\beta}}_{k,\sigma,m}=0$\,.
By Remarks \ref{TurmBemZwei} and \ref{intbarkeitTuerme} and
$$\rot E\in\Lzqpe{s+1}(A_{\hat{r}})\qqtext{resp.}\pdiv E\in\qLz{q-1}{s+1}(A_{\hat{r}})$$
we see that all coefficients $\beta_{k,\sigma,m}$ with $\sigma\leq2+s-N/2$ and $k=2,4$ resp. $k=1,4$ vanish
except of $\beta_{4,0,1}$ (for $q=0$), $\beta_{1,1,1}$ (for $q=1$), $\beta_{2,1,1}$ (for $q=N-1$) and $\beta_{4,0,1}$ (for $q=N$).

Let us discuss the case $\pdiv E\in\qLz{q-1}{s+1}(A_{\hat{r}})$\,.
Then we get by Remark \ref{TurmBemZwei} in $A_{\hat{r}}$
\begin{align*}
&\qquad\qquad E-\sum_{\substack{k,\sigma,m\,,\\\sigma>2+s-N/2}}\beta_{k,\sigma,m}\cdot Q^{q,k}_{\sigma,m}\\
&=\sum_{\substack{k,\sigma,m\,,\\k=2,3\,,\\\sigma\leq2+s-N/2}}\beta_{k,\sigma,m}\cdot Q^{q,k}_{\sigma,m}
+\begin{cases}
\beta_{4,0,1}\cdot Q^{0,4}_{0,1}&\,,\,q=0\\
\beta_{1,1,1}\cdot Q^{1,1}_{1,1}&\,,\,q=1\,\wedge s\geq N/2-1\,\\
0&\,,\,\text{otherwise}
\end{cases}\\
&=\quad\sum_{I\in\pcIbq{\leq1}{s}} e^{q,2}_I\cdot D^q_I+\check{e}^{q,2}\cdot\check{D}^{q,2}_s\qquad.
\end{align*}
Because $\bds\sum_{\substack{k,\sigma,m\,,\\\sigma>2+s-N/2}}\beta_{k,\sigma,m}\cdot\eta Q^{q,k}_{\sigma,m}\in\Lzqsom \eds$
we obtain
$$\eta E-\sum_{I\in\pcIbq{\leq1}{s}} e^{q,2}_I\cdot\eta D^q_I-\check{e}^{q,2}\cdot\eta\check{D}^{q,2}_s\in\Lzqsom$$
and thus $E\in\Lzqsom\boxplus\eta\calD^q(\pcIbq{\leq1}{s})\boxplus\eta\check{\calD}^{q,2}_s$\,.

The other two cases may be shown in the same way.
\end{proof}

We recall from \cite[p. 1034]{sphharm} the set of exceptional weights
\beq\pI:=\setb{n+N/2}{n\in\nzn}\cup\setb{1-n-N/2}{n\in\nzn}\qquad.\mylabel{pIdef}\eeq
(There this set is denoted by $J$ and here we only need it in the Hilbert space case $p=2$\,.)
Moreover, we define for $s>1-N/2$
\beq\bDqsom:=\dqsnom\cap\dhqom^\bot\qqtext{,}\bRqsom:=\ronqsnom\cap\dhqom^\bot\mylabel{bDbRdef}\eeq
and put as usual $\bDom{q}{}{0}:=\bDom{q}{0}{0}$ and $\bRom{q}{}{0}:=\bRom{q}{0}{0}$\,.
Now we are ready to establish our electro-magneto static solution theory and
formulate a first result for homogeneous, isotropic media:

\begin{lem}\mylabel{gewstatikdiv}
Let $1-N/2<s\notin\pI$\,. Then
$$\Abb{\DIV^{q+1}_{s-1}}{D(\DIV^{q+1}_{s-1})}{\bDqsom}{H}{\pdiv H}$$
is a continuous and surjective Fredholm operator on its domain of definition
\begin{align*}
D(\DIV^{q+1}_{s-1})&:=\Big(\big(\ronqpeom{s-1}\cap\dqpeom{s-1}\big)\boxplus\eta\calR^{q+1}(\pcJbqpe{0}{s-1})\boxplus\eta\check{\calR}^{q+1,1}_{s-1}\Big)\cap\ronqpelocnom\\
&\,\subset\,\ronqpenom{>-\Nh}\cap\dqpeom{>-\Nh}
\end{align*}
with kernel $\dhqpeom$\,.
\end{lem}

\begin{proof}
Let us abbreviate $\DIV:=\DIV^{q+1}_{s-1}$\,.
$\pdiv$ and $\rot$ map tower-forms from
$$\eta\calR^{q+1}(\pcJbqpe{0}{s-1})\boxplus\eta\check{\calR}^{q+1,1}_{s-1}$$
to compactly supported forms. Thus with Corollary \ref{dirichletsenkrecht} and Remark \ref{intbarkeitTuerme}
$\DIV$ is well defined, linear and continuous because $\eta\calR^{q+1}(\pcJbqpe{0}{s-1})\boxplus\eta\check{\calR}^{q+1,1}_{s-1}$
is finite dimensional. Lemma \ref{intdirichletid} yields
$$N(\DIV)\subset\dH{q+1}{>-\Nh}{}(\Omega)=\dhqpeom\qquad.$$
On the other hand let $H\in\dhqpeom$\,. Then by Lemma \ref{deltaNullzerleg} we obtain
$$H\in\Lzqpeom{s-1}\boxplus\eta\calR^{q+1}(\pcJbqpe{0}{s-1})\boxplus\eta\check{\calR}^{q+1,1}_{s-1}$$
and therefore
$$H\in\big(\ronqpeom{s-1}\cap\dqpeom{s-1}\big)\boxplus\eta\calR^{q+1}(\pcJbqpe{0}{s-1})\boxplus\eta\check{\calR}^{q+1,1}_{s-1}\qquad,$$
which implies $H\in N(\DIV)$\,, i.e. $N(\DIV)=\dhqpeom$\,.

So only the surjectivity of $\DIV$ demands a proof. For this let $F\in\bDqsom$ and $\hat{F}$ be its
extension by zero in $\rN$\,. Applying \cite[p. 1037, Theorem 4]{sphharm} we decompose
\beq\hat{F}=:F_\textrm{D}+F_\textrm{R}+F_\calS\in\dqsn\dotplus\rqsn\dotplus\calS^q_s\mylabel{Fdach}\eeq
\big(Here the set $\calS^q_s=C_{\Delta,\eta}\Lin\set{Q^{q,4}_{\sigma,m}}{\sigma<s-N/2}$ is a finite dimensional
subspace of $\cqun\big(\rN\ohne\{0\}\big)$ and $C_{\Delta,\eta}$ denotes the commutator of the Laplacian $\Delta$
and the multiplication by $\eta$\,.\big) and set
\beq f:=F_\textrm{D}-\sum_{\substack{I\in\pcIbq{1}{s-2}}}\skp{F_\textrm{D}}{D^q_{-I}}_{\Lzq{}}\cdot C_{\Delta,\eta}D^q_I\qquad.\mylabel{kleinfdef}\eeq
The duality products are well defined because by Remark \ref{intbarkeitTuerme}
$$I\in\pcIbq{1}{s-2}\quad\equi\quad\ei(I)<s-1-N/2\quad\equi\quad D^q_{-I}=\turmd{q}{1}{\ei(I)}{\ci(I)}{+}\in\Lzq{-s}\quad.$$
Using Remark \ref{TurmBemFuenf} we get for $I\in\pcIbq{1}{}$
$$C_{\Delta,\eta}D^q_I=\pdiv\rot(\eta D^q_I)+\rot\ub{\pdiv(\eta D^q_I)}_{=0}\in\dqvoxn$$
and hence $f\in\dqsn$\,. Moreover, we compute for all $I,\tilde{I}\in\pcIbq{1}{}$ with Remark \ref{TurmBemZwei} and \cite[p. 1035, (73)]{sphharm}
\begin{align*}
\skpb{C_{\Delta,\eta}D^q_I}{D^q_{-\tilde{I}}}_{\Lzq{}}
&=\skpb{C_{\Delta,\eta}\turmd{q}{1}{\ei(I)}{\ci(I)}{-}}{\turmd{q}{1}{\ei(\tilde{I})}{\ci(\tilde{I})}{+}}_{\Lzq{}}\\
&=\frac{-1}{2\ei(\tilde{I})+N}\cdot\skpb{C_{\Delta,\eta}Q^{q,2}_{\ei(I)+1,\ci(I)}}{P^{q,2}_{\ei(\tilde{I})+1,\ci(\tilde{I})}}_{\Lzq{}}\\
&=\ub{\frac{-\alpha\big(N,\ei(I)+1\big)}{2\ei(\tilde{I})+N}}_{=1}\,\cdot\,\delta_{\ei(I),\ei(\tilde{I})}\cdot\delta_{\ci(I),\ci(\tilde{I})}\qquad.
\end{align*}
Thus for all $I\in\pcIbq{1}{s-2}$ we have
$$\skp{f}{D^q_{-I}}_{\Lzq{}}=0$$
and finally by Remark \ref{TurmBemZwei}
$$f\in\dqsn\cap\big(\calP^{q,2}_{<s-\Nh}\big)^\bot\qquad.$$
In particular $f$ belongs to the range of the operator $B$ from \cite[p. 1039, Theorem 7]{sphharm},
if $1\leq q\leq N-1$ or $s\leq N/2$\,. To get this likewise in the case $q=0$ and $s>N/2$\,,
additionally $f$ has to be perpendicular to $\Eins\cong\turmd{0}{0}{0}{1}{+}$\,. This can be achieved
replacing $f$ in \eqref{kleinfdef} by
$$\tilde{f}:=f-\skp{F_\textrm{D}}{\turmd{0}{0}{0}{1}{+}}_{\qLz{0}{}}\cdot C_{\Delta,\eta}\turmd{0}{2}{0}{1}{-}\qquad.$$
Utilizing \cite[p. 1039, Theorem 7]{sphharm} we then obtain some
$$h\in\dqpe{s-1}\cap\rqpen{s-1}\qqtext{solving}\pdiv h=f\qquad.$$
By regularity {\paulydissregtheoaussenzwei} we even have $h\in\qh{1}{q+1}{s-1}{}$\,.
Thus the ansatz
\beq H:=\eta\cdot h+\Phi\mylabel{Hansatz}\eeq
transforms the system under consideration
$$\pdiv H=F\qqtext{,}\rot H=0$$
with our assumptions and Corollary \ref{dirichletsenkrecht} into the system
\begin{align}\begin{split}
\rot\Phi&=-\rot(\eta h)=-C_{\rot,\eta}h\in\bRom{q+2}{\vox}{0}\qquad,\\
\pdiv\Phi&=F-\pdiv(\eta h)\in\bDqsom\subset\bDom{q}{>1-\Nh}{0}\qquad.
\end{split}\mylabel{transsys}\end{align}
So to solve this system using the classical solution theory developed in \cite{potential} we only have to show $F-\pdiv(\eta h)\in\Lzqom{}$\,.
For $1\leq q\leq N-1$ we have in $\Omega$
$$F-\pdiv(\eta h)=\pdiv\big((1-\eta)h\big)+F_\textrm{R}+F_\calS
+\sum_{\substack{I\in\pcIbq{1}{s-2}}}\skp{F_\textrm{D}}{D^q_{-I}}_{\Lzq{}}\cdot C_{\Delta,\eta}D^q_I$$
and therefore $F-\pdiv(\eta h)-F_\textrm{R}\in\Lzqvoxom$\,.
This remains true even in the case $q=0$\,, $s>N/2$\,.
Furthermore, \eqref{Fdach} and the vanishing divergence of $F$ yield
$$F_\textrm{R}\in\rqsn\qqtext{and}\pdiv F_\textrm{R}=0\qquad\text{in }\Omega\ohne\supp F_\calS\qquad.$$
By Theorem \ref{entwicklungssatzganzraum} and Remark \ref{intbarkeitTuerme} or as in the proof of
Lemma \ref{intdirichletid} we obtain
$$F_\textrm{R}\in\Lzqom{<\Nh-1}\subset\Lzqom{}\qquad.$$
Now we are able to apply \cite[Satz 2]{potential} or
{\paulydisstopiso} and get some
$$\Phi\in\ronqpeom{-1}\cap\dqpeom{-1}$$
solving the system \eqref{transsys}. Moreover, in $A_{r_2}$
$$\rot\Phi=0\qqtext{,}\rot\pdiv\Phi=\rot F_\textrm{R}$$
hold and thus
$$\restr{\Delta\Phi}{A_{r_2}}=0\qqtext{,}(\pdiv\Phi,\rot\Phi)\in\Lzqsom\times\qLzom{q+2}{\vox}\qquad.$$
Lemma \ref{deltaNullzerleg} yields $\Phi\in\Lzqpeom{s-1}\boxplus\eta\calR^{q+1}(\pcJbqpe{0}{s-1})\boxplus\eta\check{\calR}^{q+1,1}_{s-1}$ and hence
$$\Phi\in\big(\ronqpeom{s-1}\cap\dqpeom{s-1}\big)\boxplus\eta\calR^{q+1}(\pcJbqpe{0}{s-1})\boxplus\eta\check{\calR}^{q+1,1}_{s-1}\qquad.$$
This shows $H\in D(\DIV)$ and $\pdiv H=F$\,, which completes the proof.
\end{proof}

\begin{lem}\mylabel{gewstatikrot}
Let $1-N/2<s\notin\pI$\,. Then
$$\Abb{\ROT^q_{s-1}}{D(\ROT^q_{s-1})}{\bRqpesom}{E}{\rot E}$$
is a continuous and surjective Fredholm operator on its domain of definition
\begin{align*}
D(\ROT^q_{s-1})&:=\Big(\big(\ronqom{s-1}\cap\dqom{s-1}\big)\boxplus\eta\calD^q(\pcIbq{0}{s-1})\boxplus\eta\check{\calD}^{q,1}_{s-1}\Big)\cap\dqlocnom\\
&\,\subset\,\ronqom{>-\Nh}\cap\dqnom{>-\Nh}
\end{align*}
with kernel $\dhqom$\,.
\end{lem}

\begin{proof}
The proof is analogous to the last one but more simple because the extension by zero into $\rN$ of $G\in\bRqpesom$
is still an element of $\rqpesn$\,, such that we do not need a Helmholtz decomposition like in \eqref{Fdach}.
The roles of $D^q_I$ are now played by $R^{q+1}_J$\,, $J\in\pcJbqpe{1}{s-2}$\,, and we use
\cite[p. 1037, Theorem 5]{sphharm} instead of \cite[p. 1039, Theorem 7]{sphharm}.
In the special case $q=N-1$ \big(formerly $q=0$\big), $s>N/2$ we have to guarantee the orthogonality
to $*\Eins\cong\turmr{N}{0}{0}{1}{+}$ with the help of $\turmr{N}{2}{0}{1}{-}$\,.
\end{proof}

We can generalize Lemma \ref{deltaNullzerleg} to

\begin{lem}\mylabel{rotdivintbarzerleg}
Let $\hat{r}\geq r_0$ and $E\in\Lzqom{-\Nh}$\,. If
$$\pdiv E\in\qLz{q-1}{s+1}(A_{\hat{r}})\qqtext{and}\rot E\in\Lzqpe{s+1}(A_{\hat{r}})$$
hold with some $-N/2<s\notin\pI$\,, then
$$E\in\Lzqsom\boxplus\eta\calD^q(\pcIbq{0}{s})\boxplus\eta\check{\calD}^{q,1}_s
=\Lzqsom\boxplus\eta\calR^q(\pcJbq{0}{s})\boxplus\eta\check{\calR}^{q,1}_s\qquad.$$
\end{lem}

\begin{proof}
Let $\varphi:=\mbox{\boldmath$\eta$}\big(r/(2\hat{r})\big)$\,. Then
$\varphi E\in\ronq{-\Nh}(A_{r_0})\cap\pdq{-\Nh}(A_{r_0})$ and
\begin{align*}
\pdiv(\varphi E)&\in\bD{q-1}{s+1}{0}(A_{r_0})\qquad,\\
\rot(\varphi E)&\in\bR{q+1}{s+1}{0}(A_{r_0})\qquad.
\end{align*}
By Lemma \ref{gewstatikdiv} and Lemma \ref{gewstatikrot} there exists some
$$e\in\big(\ronqs(A_{r_0})\cap\dqs(A_{r_0})\big)\boxplus\eta\calD^q(\pcIbq{0}{s})\boxplus\eta\check{\calD}^{q,1}_s\qquad,$$
such that $\rot e=\rot(\varphi E)$ and $\pdiv e=\pdiv(\varphi E)$\,.
Thus $e-\varphi E\in\dH{q}{-\Nh}{}(A_{r_0})$ is a Dirichlet form and therefore $e-\varphi E\in\dH{q}{}{}(A_{r_0})$
by Lemma \ref{intdirichletid}. Extending $e$ by zero into $\Omega$ we get with Lemma \ref{deltaNullzerleg}
$e,e-\varphi E\in\Lzqsom\boxplus\eta\calD^q(\pcIbq{0}{s})\boxplus\eta\check{\calD}^{q,1}_s$\,. Thus $\varphi E$
and $E=(1-\varphi)E+\varphi E$ are elements of $\Lzqsom\boxplus\eta\calD^q(\pcIbq{0}{s})\boxplus\eta\check{\calD}^{q,1}_s$\,.
\end{proof}

Now we consider inhomogeneous, anisotropic media. First we want to generalize Lemma \ref{intdirichletid}:

\begin{lem}\mylabel{intdirichleteps}
Let $\tau>0$\,. Then
$$\dhqepsom{-\Nh}=\dhqepsom{}=\dhqepsom{<\Nh-1}\qquad.$$
For $q\notin\{1,N-1\}$ even $\dhqepsom{}=\dhqepsom{<\Nh}$ holds.
\end{lem}

\begin{rem}\mylabel{intdirichletepsbem}
In particular $\dhqepsom{}\subset\Lzqom{-s}$\,, if $s>1-N/2$\,. Moreover, in the case $q\notin\{1,N-1\}$
this inclusion remains valid for $s>-N/2$\,.
\end{rem}

\begin{proof}
Let $E\in\dhqepsom{-\Nh}$\,. By regularity, e.g. {\paulydissregcoraussen}, $E$ belongs to
$\qh{1}{q}{-\Nh}{}(A_{r_0})$ and thus in $A_{r_0}$
$$\rot E=0\qqtext{,}\pdiv E=-\pdiv\epsd E\in\qLz{q-1}{-\Nh+1+\tau}(A_{r_0})$$
hold true. We assume w. l. o. g. $\tau-N/2\notin\pI$ and obtain by Lemma \ref{rotdivintbarzerleg}
$$E\in\Lzqom{\tau-\Nh}\boxplus\eta\calD^q(\pcIbq{0}{\tau-\Nh})\boxplus\eta\check{\calD}^{q,1}_{\tau-\Nh}\qquad.$$
By Remark \ref{intbarkeitTuerme} we get
$$\eta\calD^q(\pcIbq{0}{\tau-\Nh})\boxplus\eta\check{\calD}^{q,1}_{\tau-\Nh}\subset\Lzqom{<s_q}$$
with $s_q:=N/2-\delta_{q,1}-\delta_{q,N-1}$\,.
If $\tau-N/2\geq s_q$\,, we get $E\in\Lzqom{<s_q}$\,, i.e. $E\in\dhqepsom{<s_q}$\,.
If $\tau-N/2<s_q$\,, we only have $E\in\Lzqom{\tau-\Nh}$\,, i.e. $E\in\dhqepsom{\tau-\Nh}$\,.
Repeating this argument leads us after finitely many $\tau$-steps to $E\in\dhqepsom{<s_q}$\,.
\end{proof}

Using Helmholtz decompositions, e.g. \paulytimeharmlzzerl, it is easy to show that the dimension $d^q$ of
the Dirichlet forms $\dhqepsom{}$ does not depend on the transformation $\eps$\,, i.e.
$$d^q=\dim\dhqom\qquad.$$
Furthermore, $d^q$ is finite since $\om$ has the {\sf MLCP}.
Moreover, we obtain

\begin{cor}\mylabel{dimdirichleteps}
Let $\tau>0$ and $-N/2\leq t<N/2-1$\,. Then
$$\dim\dhqepsom{t}=d^q<\infty\qquad.$$
If $q\notin\{1,N-1\}$ the first equation holds even for $-N/2\leq t<N/2$\,.
\end{cor}

Lemma \ref{intdirichleteps} yields a generalization of Corollary \ref{dirichletsenkrecht}:

\begin{cor}\mylabel{dirichletsenkrechteps}
Let $\tau>0$ and $s>1-N/2$\,. Then with closures in $\Lzqsom$
\begin{align*}
\ol{\rot\pr{q-1}{s-1}{}{\circ}(\Omega)}\cup\ol{\rot\pR{q-1}{s}{}{\circ}(\Omega)}&\subset\dhqepsom{}^{\bot_\eps}\qquad,\\
\ol{\pdiv\pdi{q+1}{s-1}{}{}(\Omega)}\cup\ol{\pdiv\pDi{q+1}{s}{}{}(\Omega)}&\subset\dhqepsom{}^\bot
\end{align*}
hold. Here we denote by $\bot_\eps$ the orthogonality with respect to the $\skp{\eps\,\cdot\,}{\,\cdot\,}_{\Lzqom{}}$-duality.
\end{cor}

\begin{lem}\mylabel{gewstatikdivmu}
Let $1-N/2<s\notin\pI$ and $\tau>\max\{0,s-N/2\}$\,, $\tau\geq-s$\,. Then
$$\Abb{\sideset{_\mu}{^{q+1}_{s-1}}\DIV}{D(\sideset{_\mu}{^{q+1}_{s-1}}\DIV)}{\bDqsom}{H}{\pdiv H}$$
is a continuous and surjective Fredholm operator on its domain of definition
\begin{align*}
&\qquad D(\sideset{_\mu}{^{q+1}_{s-1}}\DIV)\\
&:=\Big(\big(\mu^\me\ronqpeom{s-1}\cap\dqpeom{s-1}\big)\boxplus\eta\calR^{q+1}(\pcJbqpe{0}{s-1})\boxplus\eta\check{\calR}^{q+1,1}_{s-1}\Big)\cap\mu^\me\ronqpelocnom\\
&\,\subset\,\mu^\me\ronqpenom{>-\Nh}\cap\dqpeom{>-\Nh}
\end{align*}
with kernel $\mu^\me\dH{q+1}{}{\mu^\me}(\Omega)$\,.
\end{lem}

\begin{proof}
We set $\DIV:=\sideset{_\mu}{^{q+1}_{s-1}}\DIV$ and follow the proof of Lemma \ref{gewstatikdiv}.

A form $\eta H\in\eta\calR^{q+1}(\pcJbqpe{0}{s-1})\boxplus\eta\check{\calR}^{q+1,1}_{s-1}$ belongs to $\qh{1}{q+1}{<\Nh-1}{}$
and the assumptions on $\tau$ yield
$$\rot(\mu\eta H)=C_{\rot,\eta}H+\rot(\mud\eta H)\in\qLzom{q+2}{<\Nh+\tau}\subset\qLzom{q+2}{s}\qquad.$$
Thus $\DIV$ is well defined and clearly linear and continuous. By Lemma \ref{intdirichleteps} we obtain $\mu N(\DIV)\subset\dH{q+1}{}{\mu^\me}(\Omega)$\,.
Applying the regularity result {\paulydissregcoraussen} we achieve $H\in\mu^\me\dH{q+1}{}{\mu^\me}(\Omega)\subset\qh{1}{q+1}{<\Nh-1}{}(A_{r_1})$ and therefore
$$\pdiv H=0\qtext{,}\rot H=-\rot(\mud H)\in\qLz{q+2}{<\Nh+\tau}(A_{r_2})\subset\qLz{q+2}{s}(A_{r_2})\quad,$$
which implies $\mu^\me\dH{q+1}{}{\mu^\me}(\Omega)\subset N(\DIV)$ by Lemma \ref{rotdivintbarzerleg}.

So it remains to show that $\DIV$ is surjective. Let $F\in\bDqsom$\,.
We follow exactly the arguments in Lemma \ref{gewstatikdiv} leading to the ansatz \eqref{Hansatz}.
By Corollary \ref{dirichletsenkrecht} the system
$$\pdiv H=F\qqtext{,}\rot\mu H=0$$
is transformed into
\begin{align}\begin{split}
\rot\mu\Phi&=-\rot(\mu\eta h)=-C_{\rot,\eta}h-\rot(\mud\eta h)\in\bRom{q+2}{s+\tau}{0}\qquad,\\
\pdiv\Phi&=F-\pdiv(\eta h)\in\bDqsom\subset\bDom{q}{>1-\Nh}{0}\qquad.
\end{split}\mylabel{rotmudivPhi}\end{align}
As in the proof of Lemma \ref{gewstatikdiv} we compute $F-\pdiv(\eta h)\in\Lzqom{<\Nh-1}$ and with $\tau\geq-s$ we get additionally
\beq\big(F-\pdiv(\eta h),-\rot(\mu\eta h)\big)\in\Lzqom{}\times\qLzom{q+2}{}\qquad.\mylabel{klassWerteb}\eeq
Thus the generalized classical static solution theory from {\paulydisstopiso} yields some
$$\Phi\in\mu^\me\ronqpeom{-1}\cap\dqpeom{-1}$$
solving the system \eqref{rotmudivPhi}.
Clearly $\eta h\in\qh{1}{q+1}{s-1}{}$ implies $\eta h\in D(\DIV)$\,. So our proof is complete, if we can show
$\Phi\in D(\DIV)$\,. But because of $\pdiv\Phi\in\Lzqsom$ this decomposition of $\Phi$ follows by Lemma \ref{rotdivintbarzerleg}
and the assumptions on $\tau$\,, if e.g.
\beq\rot\Phi\in\qLz{q+2}{s}(A_{r_1})\qquad.\mylabel{rotPhilzs}\eeq
By regularity, e.g. {\paulydissregcoraussen}, $\Phi\in\qh{1}{q+1}{-1}{}(A_{r_1})$\,,
i.e. $\rot\Phi\in\qLz{q+2}{}(A_{r_1})$\,. Thus we may assume $s>0$ in \eqref{rotPhilzs}. Because of
\beq\rot\Phi=-\rot(\mud\Phi)-\rot(\mu\eta h)\in\qLz{q+2}{\min\{\tau,s+\tau\}}(A_{r_1})=\qLz{q+2}{\tau}(A_{r_1})\mylabel{rotPhi}\eeq
we only have to discuss the case $0<\tau<s$\,. From Lemma \ref{rotdivintbarzerleg} (w. l. o. g. $\tau\notin\pI$) we obtain
$$\Phi\in\Lzqpeom{\tau-1}\boxplus\eta\calR^{q+1}(\pcJbqpe{0}{\tau-1})\boxplus\eta\check{\calR}^{q+1,1}_{\tau-1}\qquad.$$
If $N/2\leq\tau<s$ we get $\Phi\in\Lzqpeom{<\Nh-1}$ and with \eqref{rotPhi} $\rot\Phi\in\qLz{q+2}{\Nh}(A_{r_1})$\,, i.e.
$\Phi\in\qh{1}{q+1}{<\Nh-1}{}(A_{r_1})$\,. \eqref{rotPhi} and the assumption $\tau>s-N/2$ show \eqref{rotPhilzs}.
In the other case $\tau<\min\{s,N/2\}$ we have $\eta\calR^{q+1}(\pcJbqpe{0}{\tau-1})\boxplus\eta\check{\calR}^{q+1,1}_{\tau-1}=\{0\}$ and thus $\Phi\in\Lzqpeom{\tau-1}$\,.
Once more we obtain $\Phi\in\qh{1}{q+1}{\tau-1}{}(A_{r_1})$ and with \eqref{rotPhi} $\rot\Phi\in\qLz{q+2}{\min\{2\tau,s+\tau\}}(A_{r_1})$\,.
After finitely many repetitions of this argument either $\ell\tau\geq s$ or $\ell\tau\geq N/2$ with $\ell\in\nz$ holds.
By the arguments given above we achieve \eqref{rotPhilzs} in this case as well.
\end{proof}

\begin{cor}\mylabel{gewstatikdivmukor}
Let $1-N/2<s\notin\pI$ and $\tau>\max\{0,s-N/2\}$\,, $\tau\geq-s$\,. Then
\begin{align*}
&\qquad\mu^\me D(\sideset{_{\mu^\me}}{^{q+1}_{s-1}}\DIV)\\
&=\Big(\big(\ronqpeom{s-1}\cap\mu^\me\dqpeom{s-1}\big)\boxplus\eta\calR^{q+1}(\pcJbqpe{0}{s-1})\boxplus\eta\check{\calR}^{q+1,1}_{s-1}\Big)\cap\ronqpelocnom
\end{align*}
and with $D(\sideset{^\mu}{^{q+1}_{s-1}}\DIV):=\mu^\me D(\sideset{_{\mu^\me}}{^{q+1}_{s-1}}\DIV)$
$$\Abb{\sideset{^\mu}{^{q+1}_{s-1}}\DIV}{D(\sideset{^\mu}{^{q+1}_{s-1}}\DIV)}{\bDqsom}{H}{\pdiv\mu H}$$
is a continuous and surjective Fredholm operator with kernel $\dH{q+1}{}{\mu}(\Omega)$\,.
\end{cor}

\begin{proof}
The assertions follow from the previous lemma, if we can show the first assertion of this corollary.
But with the properties of $\tau$ this is clear because of
\begin{align*}
&\qquad\mu^\me\big(\eta\calR^{q+1}(\pcJbqpe{0}{s-1})\boxplus\eta\check{\calR}^{q+1,1}_{s-1}\big)\\
&\subset\big(\ronqpeom{s-1}\cap\mu^\me\dqpeom{s-1}\big)\boxplus\eta\calR^{q+1}(\pcJbqpe{0}{s-1})\boxplus\eta\check{\calR}^{q+1,1}_{s-1}\qquad.
\end{align*}
\end{proof}

Analogously we obtain

\begin{lem}\mylabel{gewstatikroteps}
Let $1-N/2<s\notin\pI$ and $\tau>\max\{0,s-N/2\}$\,, $\tau\geq-s$\,. Then
$$\Abb{\sideset{_\eps}{^q_{s-1}}\ROT}{D(\sideset{_\eps}{^q_{s-1}}\ROT)}{\bRqpesom}{E}{\rot E}$$
is a continuous and surjective Fredholm operator on its domain of definition
\begin{align*}
&\qquad D(\sideset{_\eps}{^q_{s-1}}\ROT)\\
&:=\Big(\big(\ronqom{s-1}\cap\eps^\me\dqom{s-1}\big)\boxplus\eta\calD^q(\pcIbq{0}{s-1})\boxplus\eta\check{\calD}^{q,1}_{s-1}\Big)\cap\eps^\me\dqlocnom\\
&\,\subset\,\ronqom{>-\Nh}\cap\eps^\me\dqnom{>-\Nh}
\end{align*}
with kernel $\dhqepsom{}$\,.
\end{lem}

\begin{cor}\mylabel{gewstatikrotepskor}
Let $1-N/2<s\notin\pI$ and $\tau>\max\{0,s-N/2\}$\,, $\tau\geq-s$\,. Then
\begin{align*}
&\qquad \eps^\me D(\sideset{_{\eps^\me}}{^q_{s-1}}\ROT)\\
&=\Big(\big(\eps^\me\ronqom{s-1}\cap\dqom{s-1}\big)\boxplus\eta\calD^q(\pcIbq{0}{s-1})\boxplus\eta\check{\calD}^{q,1}_{s-1}\Big)\cap\dqlocnom
\end{align*}
and with $D(\sideset{^\eps}{^q_{s-1}}\ROT):=\eps^\me D(\sideset{_{\eps^\me}}{^q_{s-1}}\ROT)$
$$\Abb{\sideset{^\eps}{^q_{s-1}}\ROT}{D(\sideset{^\eps}{^q_{s-1}}\ROT)}{\bRqpesom}{E}{\rot\eps E}$$
is a continuous and surjective Fredholm operator with kernel $\eps^\me\dH{q}{}{\eps^\me}(\Omega)$\,.
\end{cor}

\begin{rem}\mylabel{gewstatikWertebeins}
Let the assumptions of Lemma \ref{gewstatikdivmu} be fulfilled
and additionally $\tilde{\eps}$\,, $\tilde{\mu}$ be
two $\tilde{\tau}$-$\pc{1}$-admissible transformations with $\tilde{\tau}>0$\,.
Then we can characterize the ranges of $\sideset{_\mu}{^{q+1}_{s-1}}\DIV$\,,
$\sideset{^\mu}{^{q+1}_{s-1}}\DIV$ resp. $\sideset{_\eps}{^q_{s-1}}\ROT$\,, $\sideset{^\eps}{^q_{s-1}}\ROT$ by
$$\bDqsom=\dqsnom\cap\dH{q}{}{\tilde{\eps}}(\Omega)^\bot\qtext{resp.}
\bRqpesom=\ronqpesnom\cap\dH{q+1}{}{\tilde{\mu}}(\Omega)^{\bot_{\tilde{\mu}}}$$
as well.
\end{rem}

\begin{proof}
By Corollary \ref{dirichletsenkrechteps} all operators are still well defined.
Let us consider e.g. $\sideset{_\mu}{^{q+1}_{s-1}}\DIV$ from Lemma \ref{gewstatikdivmu}.
Only the argument showing surjectivity has to be changed. Now \eqref{rotmudivPhi} and \eqref{klassWerteb} are replaced by
$$\big(F-\pdiv(\eta h),-\rot(\mu\eta h)\big)\in\big(\dqnom{}\cap\dH{q}{}{\tilde{\eps}}(\Omega)^\bot\big)\times\big(\pr{q+2}{}{0}{\circ}(\Omega)\cap\dH{q+2}{}{}(\Omega)^\bot\big)$$
but with {\paulytimeharmlzzerl} we see that the latter set equals
$$\ol{\pdiv\Dqpeom{}}\times\ol{\rot\Ronqpeom{}}$$
and thus is independent of $\tilde{\eps}$\,.

Similarly we prove the assertion concerning the range of $\sideset{_\eps}{^q_{s-1}}\ROT$\,.
\end{proof}

\section{Generalized electro-magneto statics}

For $s>1-N/2$ we put \mylabel{statics}
$$\bWqsom:=\bDom{q-1}{s}{0}\times\bRqpesom\times\cz^{d^q}$$
and choose $d^q$ continuous linear functionals $\Phi^\ell_\nu$ on
$$\big(\rqom{s-1}\cap\nu^\me\dqom{s-1}\big)\boxplus\eta\calD^q(\pcIbq{0}{s-1})\boxplus\eta\check{\calD}^{q,1}_{s-1}$$
with
$$\dH{q}{}{\nu}(\Omega)\cap\bigcap_{\ell=1}^{d^q}N(\Phi^\ell_\nu)=\{0\}$$
for some given $0$-admissible transformation $\nu$\,. We set $\Phi_\nu:=(\Phi_\nu^1\,\cdot\,,\dots,\Phi_\nu^{d^q}\,\cdot\,)$ and obtain

\begin{theo}\mylabel{gewstatikinjektivrotdiv}
Let $s\in(1-N/2,\infty)\ohne\pI$ and $\tau>\max\{0,s-N/2\}$\,, $\tau\geq-s$ as well as
$$D(\sideset{_\eps}{^q_{s-1}}\MAX):=\big(\ronqom{s-1}\cap\eps^\me\dqom{s-1}\big)\boxplus\eta\calD^q(\pcIbq{0}{s-1})\boxplus\eta\check{\calD}^{q,1}_{s-1}\qquad.$$
Then the operator
$$\Abb{\sideset{_\eps}{^q_{s-1}}\MAX}{D(\sideset{_\eps}{^q_{s-1}}\MAX)}{\bWqsom}{E}{\big(\pdiv\eps E,\rot E,\Phi_\eps(E)\big)}$$
is a topological isomorphism.
\end{theo}

\begin{rem}\mylabel{gewstatikinjektivrotdivbem}
Let $\nu$ be a $0$-admissible and $\lambda$ be a $\tau$-$\pc{1}$-admissible
transformation on $q$-forms. Moreover, let
$\big\{{}_\theta h_\ell\big\}_{\ell=1}^{d^q}$ for $\theta\in\{\eps,\lambda\}$
be some basis of $\dH{q}{}{\theta}(\Omega)$\,.
Then for weights $s>2-N/2$ we can choose, for instance, the functionals
$\Phi^\ell_\eps(E):=\skp{\nu E}{{}_{\eps}h_\ell}_{\Lzqom{}}$ or
$\Phi^\ell_\eps(E):=\skp{\eps E}{{}_\lambda h_\ell}_{\Lzqom{}}$ or
$\Phi^\ell_\eps(E):=\skp{\lambda E}{{}_\lambda h_\ell}_{\Lzqom{}}$\,.
\end{rem}

\begin{cor}\mylabel{gewstatikinjektivrotdivcor}
Let the assumptions of Theorem \ref{gewstatikinjektivrotdiv} be satisfied. Then
\begin{align*}
D(\sideset{^\eps}{^q_{s-1}}\MAX)&:=\eps^\me D(\sideset{_{\eps^\me}}{^q_{s-1}}\MAX)\\
&\;=\big(\eps^\me\ronqom{s-1}\cap\dqom{s-1}\big)\boxplus\eta\calD^q(\pcIbq{0}{s-1})\boxplus\eta\check{\calD}^{q,1}_{s-1}
\end{align*}
holds and
$$\Abb{\sideset{^\eps}{^q_{s-1}}\MAX}{D(\sideset{^\eps}{^q_{s-1}}\MAX)}{\bWqsom}{E}{\big(\pdiv E,\rot\eps E,\Phi_{\eps^\me}(\eps E)\big)}$$
is a topological isomorphism.
\end{cor}

\begin{proof}
By Corollary \ref{gewstatikdivmukor} and Lemma \ref{gewstatikroteps} $\sideset{_\eps}{^q_{s-1}}\MAX$ is continuous and with our assumptions
clearly injective. Thus by the bounded inverse theorem $\sideset{_\eps}{^q_{s-1}}\MAX$ is a topological isomorphism, if it is surjective. Let
$(f,G,\gamma)\in \bWqsom$\,. Then a combination of Corollary \ref{gewstatikdivmukor} and Lemma \ref{gewstatikroteps} yields some
$\hat{E}\in D(\sideset{_\eps}{^q_{s-1}}\MAX)$ solving $\rot\hat{E}=G$ and $\pdiv\eps\hat{E}=f$
and we are free in adding a Dirichlet form from $\dhqepsom{}$ to $\hat{E}$\,. By our assumptions
$$\Abb{\phi}{\dhqepsom{}}{\cz^{d^q}}{E}{\Phi_\eps(E)}$$
is a topological isomorphism. Therefore
$E:=\hat{E}+\phi^\me\big(\gamma-\Phi_\eps(\hat{E})\big)$
is the unique solution of $\sideset{_\eps}{^q_{s-1}}\MAX E=(f,G,\gamma)$\,.
From the properties of $\tau$ we get easily
$\sideset{^\eps}{^q_{s-1}}\MAX=\sideset{_{\eps^\me}}{^q_{s-1}}\MAX\eps$\,,
such that this operator is also a topological isomorphism, which proves the corollary.

If $s>2-N/2$ we have $D(\sideset{_\eps}{^q_{s-1}}\MAX)\subset\Lzqom{>1-\Nh}$\,. Then by Lemma \ref{intdirichleteps} and Remark \ref{intdirichletepsbem}
the scalar products in Remark \ref{gewstatikinjektivrotdivbem} are well defined and possible choices for $\Phi^\ell_\eps$\,.
\end{proof}

Actually we are interested in a (electro-magneto) static solution theory suited for the operator $M$ from \eqref{Mdef}. Moreover, we want
to define higher powers of a special static solution operator. The main tool for the iteration process are the tower-forms from section \ref{sectiontower}.
(Until now essentially we only needed the ground-forms of height zero to establish our solution theory.)
Thus we expect that the heights of the tower-parts of our solutions will grow in each step of the iteration, which implies decreasing
integrability features of these solutions. But to guarantee uniqueness of the solutions, we always have to project along
the Dirichlet forms. Therefore it makes no sense to proceed with orthogonality constraints with respect to the Dirichlet forms anymore and
we are forced to work with orthogonality constraints utilizing forms with compact supports in $\rN$\,, such that the duality products with our
tower-forms are still well defined.

To this end we introduce from \cite[p. 41]{decomposition} for all $q$ finitely many smooth forms
$$\bonqom:=\big\{\bon{q}_1,\dots,\bon{q}_{d^q}\big\}\qqtext{,}\bqom:=\{b^q_1,\dots,b^q_{d^q}\}\qquad,$$
where the latter set is only defined for $q\neq1$\,, with \ul{com}p\ul{act} resp. \ul{bounded} support in $\Omega$ and the following properties:
For all $0$-admissible transformations $\nu$
\begin{itemize}
\item $\bonqom\subset\ronqnom{\vox}$ is linearly independent modulo $\ol{\rot\pR{q-1}{}{}{\circ}(\Omega)}$ and
$$\dH{q}{}{\nu}(\om)\cap\bonqom^{\bot_\nu}=\{0\}$$
holds. The orthogonal projections of $\bonqom$ in $\ronqnom{}$
along $\ol{\rot\pR{q-1}{}{}{\circ}(\Omega)}$ on $\dH{q}{}{\nu}(\om)$ form a basis of the Dirichlet forms $\dH{q}{}{\nu}(\om)$\,;
\item (for $q\neq1$) $\bqom\subset\dqnom{\vox}$ is linearly independent modulo $\ol{\pdiv\Dqpeom{}}$ and
$$\dH{q}{}{\nu}(\om)\cap\bqom^{\bot}=\{0\}$$
holds. The orthogonal projections of $\nu^\me\bqom$ in $\nu^\me\dqnom{}$ on $\dH{q}{}{\nu}(\om)$
along $\nu^\me\ol{\pdiv\Dqpeom{}}$ form a basis of the Dirichlet forms $\dH{q}{}{\nu}(\om)$\,.
\end{itemize}

To guarantee the existence of these forms \big(see \cite[p. 40]{decomposition}\big)
we need another (stronger) assumption on the boundary $\p\Omega$\,, i.e. $\Omega$ is Lipschitz homeomorphic to a smooth
exterior domain with boundary. We will call this property of $\Omega$ the `static Maxwell property' ({\sf SMP}), and this property
implies the {\sf MLCP}.

We note that the properties of $\bonqom$ and $\bqom$ are mentioned in \cite{decomposition} only in the case $\nu=\id$\,.
But using the Helmholtz decompositions {\paulytimeharmlzzerl} we
can easily show that these properties hold true in the general case as well.

From now on we may assume additionally that our exterior domain $\Omega$ has the {\sf SMP} and thus in particular the {\sf MLCP}
and that w.~l.~o.~g. for all $q$ all supports of the forms in $\bonqom$ and $\bqom$
are compact subsets of $U_{r_0}$\,. We remark by definition then
$$\supp\eta\cap\Big(\bigcup_{\ell=1}^{d^q}\supp\bon{q}_\ell\cup\bigcup_{k=1}^{d^q}\supp b^q_k\Big)=\emptyset\qquad.$$

In the following we will use these special forms $\bonqom$ and $\bqom$ to project along the Dirichlet forms.
Because of their bounded supports we clearly have for all $s\in\rz$ and with closures in $\Lzqsom$
\begin{align*}
\ol{\pdiv\pdi{q+1}{s-1}{}{}(\Omega)}\cup\ol{\pdiv\pDi{q+1}{s}{}{}(\Omega)}&\subset\bonqom^\bot\qquad,\\
\ol{\rot\pr{q-1}{s-1}{}{\circ}(\Omega)}\cup\ol{\rot\pR{q-1}{s}{}{\circ}(\Omega)}&\subset\bqom^\bot\qquad.
\end{align*}
Moreover, with closures in $\Lzqom{}$
\begin{align}
\ol{\pdiv\pDi{q+1}{}{}{}(\Omega)}=\bDom{q}{}{0}=\dqnom{}\cap\dH{q}{}{\nu}(\om)^\bot&=\dqnom{}\cap\bonqom^\bot\quad,\mylabel{wertebeins}\\
\ol{\rot\pR{q-1}{}{}{\circ}(\Omega)}=\bRom{q}{}{0}=\ronqnom{}\cap\dH{q}{}{\nu}(\om)^{\bot_\nu}&=\ronqnom{}\cap\bqom^\bot\mylabel{wertebzwei}
\end{align}
hold true. The first two equations in each term follow by {\paulytimeharmlzzerl} and one inclusion of the
third equation in each term is trivial. Hence, if we look, for example,
at the $q$-form $F\in\dqnom{}\cap\bonqom^\bot$\,, we decompose $F$ according to
the Helmholtz decomposition {\paulytimeharmlzzerl}
$$F=f+E\in\ol{\pdiv\pDi{q+1}{}{}{}(\Omega)}\oplus\dhqom$$
and obtain $F-f\in\dhqom\cap\bonqom^\bot=\{0\}$ by the properties of $\bonqom$\,.

Now we are able to characterize the ranges of our operators
$\sideset{_\mu}{^{q+1}_{s-1}}\DIV$\,, $\sideset{^\mu}{^{q+1}_{s-1}}\DIV$ and
$\sideset{_\eps}{^q_{s-1}}\ROT$\,, $\sideset{^\eps}{^q_{s-1}}\ROT$
even by orthogonality constraints on $\bonqom$ and $\bqpeom$\,.

\begin{cor}\mylabel{gewstatikWertebzwei}
Let the assumptions of Lemma \ref{gewstatikdivmu} be fulfilled. Then
$$\bDqsom=\dqsnom\cap\dhqepsom{}^\bot=\dqsnom\cap\bonqom^\bot$$
and in the case $q\neq0$
$$\bRqpesom=\ronqpesnom\cap\dH{q+1}{}{\mu}(\om)^{\bot_\mu}=\ronqpesnom\cap\bqpeom^\bot\quad.$$
\end{cor}

\begin{proof}
The first equations in each term have been shown in Remark \ref{gewstatikWertebeins}. To prove the second equations in each term
we use the same arguments as in the proof of Remark \ref{gewstatikWertebeins} combined with \eqref{wertebeins} and \eqref{wertebzwei}.
Thus all sets from above are just different characterizations of the ranges of
$\sideset{_\mu}{^{q+1}_{s-1}}\DIV$ or $\sideset{_\eps}{^q_{s-1}}\ROT$\,.
\end{proof}

\begin{rem}\mylabel{gewstatikinjektivrotdivbemtwo}
Looking once more at Theorem \ref{gewstatikinjektivrotdiv} and Corollary \ref{gewstatikinjektivrotdivcor}
we may represent the range of $\sideset{_\eps}{^q_{s-1}}\MAX$ and $\sideset{^\eps}{^q_{s-1}}\MAX$
with the help of Corollary \ref{gewstatikWertebzwei} in a different manner. Furthermore, for example,
$$\Phi^\ell_\eps(E):=\skp{\eps E}{\bon{q}_\ell}_{\Lzqom{}}\qqtext{or}\Phi^\ell_\eps(E):=\skp{E}{b^{q}_\ell}_{\Lzqom{}}$$
are good choices for $\Phi^\ell_\eps$\,, where the latter is only defined for $q\neq1$\,.
\end{rem}

Using the special forms $\bqpeom$ we have to restrict our considerations from now on to the case $q\neq0$\,.

The latter theorem and the corresponding remarks and corollaries yield by specialization the following electro-magneto static result,
which meets our needs and uses only the forms $\bonqom$ and $\bqom$ instead of $\dhqom$\,:

\begin{theo}\mylabel{gewstatikinjektivrotdivspeziell}
Let $q\neq0$\,, $s\in(1-N/2,\infty)\ohne\pI$ and $\tau>\max\{0,s-N/2\}$\,, $\tau\geq-s$\,. Then the operators
$$\Abb{\sideset{_\eps}{^q_{s-1}}\statMax}{D(\sideset{_\eps}{^q_{s-1}}\MAX)}{\bWqsom}{E}{\Big(\pdiv\eps E,\rot E,\big(\skp{\eps E}{\bon{q}_\ell}_{\Lzqom{}}\big)_{\ell=1}^{d^q}\Big)}\qquad,$$
$$\Abb{\sideset{^\mu}{^{q+1}_{s-1}}\statMax}{D(\sideset{^\mu}{^{q+1}_{s-1}}\MAX)}{\bWqpesom}{H}{\Big(\pdiv H,\rot\mu H,\big(\skp{\mu H}{b^{q+1}_\ell}_{\Lzqpeom{}}\big)_{\ell=1}^{d^{q+1}}\Big)}$$
are topological isomorphisms.
\end{theo}

\section{Powers of a static Maxwell operator}

From now on we only work with the forms $\bonqom$ and $\bqpeom$ since they have bounded supports and
thus we may assume $q\neq0$\,.
Then for arbitrary $s\in\rz$ and $t\in\{\loc,s\}$ the spaces
$$\bDom{q}{t}{0}=\dqtnom\cap\bonqom^\bot\qqtext{,}\bRom{q+1}{t}{0}=\ronqpetnom\cap\bqpeom^\bot$$
are well defined.

In this section we want to define powers of a special static solution operator from Theorem \ref{gewstatikinjektivrotdivspeziell},
which acts on special data $\big((0,G,0),(F,0,0)\big)\in \bWqsom\times \bWqpesom$\,, i.e.
\begin{align*}
\FG&\in\bDqsom\times\bRqpesom\qqtext{,}s>1-N/2\qquad,
\intertext{and maps onto solutions}
(\eps E,\mu H)&\in\bDom{q}{>-\Nh}{0}\times\bRom{q+1}{>-\Nh}{0}\qquad.
\end{align*}
To this end we first study each one of these two operators
$F\mapsto\mu H$ and $G\mapsto\eps E$
separately. Keeping in mind that the interesting case of the classical electro-magneto
static theory is $q=1$\,, we restrict our considerations in this section generally to ranks
$$1\leq q\leq N-2$$
to avoid the discussion of some exceptional cases, which would increase the complexity of notations in this section considerably.

A further specialization of Theorem \ref{gewstatikinjektivrotdivspeziell} shows \mylabel{iteration}

\begin{theo}\mylabel{gewstatikinjektivrotdivnull}
Let $s\in(1-N/2,\infty)\ohne\pI$ and $\tau>\max\{0,s-N/2\}$\,, $\tau\geq-s$\,. Then
\begin{align*}
&\Abb{\sideset{_\eps}{^q_{s-1}}\crot}{D(\sideset{_\eps}{^q_{s-1}}\crot)}{\bRqpesom}{E}{\rot E}\qquad,\\
&\Abb{\sideset{_\mu}{^{q+1}_{s-1}}\cdiv}{D(\sideset{_\mu}{^{q+1}_{s-1}}\cdiv)}{\bDqsom}{H}{\pdiv H}
\end{align*}
are topological isomorphisms on their domains of definition
\begin{align*}
D(\sideset{_\eps}{^q_{s-1}}\crot)&:=\Big(\big(\ronqom{s-1}\cap\eps^\me\dqom{s-1}\big)\boxplus\eta\calD^q(\pcIbq{0}{s-1})\Big)\cap\eps^\me\bDqlocom\quad,\\
D(\sideset{_\mu}{^{q+1}_{s-1}}\cdiv)&:=\Big(\big(\mu^\me\ronqpeom{s-1}\cap\dqpeom{s-1}\big)\boxplus\eta\calR^{q+1}(\pcJbqpe{0}{s-1})\Big)\cap\mu^\me\bRqpelocom\quad.
\end{align*}
\end{theo}

\begin{rem}\mylabel{gewstatikinjektivrotdivnullbem}
The exceptional forms $\eta\check{D}^{q,1}_{s-1}$ and $\eta\check{R}^{q+1,1}_{s-1}$ do no longer occur for those values of $q$\,, since
$$D(\sideset{_\eps}{^q_{s-1}}\crot)\subset\eps^\me\bDqlocom\qqtext{and}D(\sideset{_\mu}{^{q+1}_{s-1}}\cdiv)\subset\mu^\me\bRqpelocom\qquad.$$
Because of the restriction for the ranks $q$ we only have to show that the exceptional forms do not appear
in the case $q=1$ for $\sideset{_\eps}{^q_{s-1}}\crot$ and in the case $q=N-2$ for $\sideset{_\mu}{^{q+1}_{s-1}}\cdiv$\,.
The proof of these facts will be supplied in the appendix.
\end{rem}

Using these two operators we define a static solution operator $\loesn$ acting on
$$\bDqsom\times\bRqpesom$$
by
$$\loesn\FG:=\big((\sideset{_\eps}{^q_{s-1}}\crot)^\me G,(\sideset{_\mu}{^{q+1}_{s-1}}\cdiv)^\me F\big)\qquad.$$
Because the inverses $\loes_{\rot,\eps}:=\eps(\sideset{_\eps}{^q_{s-1}}\crot)^\me$
and $\loes_{\pdiv,\mu}:=\mu(\sideset{_\mu}{^{q+1}_{s-1}}\cdiv)^\me$
have mutually related domains of definition and ranges $\bDom{q}{t}{0}$ and $\bRom{q+1}{t}{0}$
we hope that $\loes_{\rot,\eps}\loes_{\pdiv,\mu}$ and $\loes_{\pdiv,\mu}\loes_{\rot,\eps}$ exist in some sense.
To this end it is necessary to generalize $\loes_{\pdiv,\mu}$ and $\loes_{\rot,\eps}$\,, such that they can be applied to tower-forms.

Before we proceed and formulate our next lemma we need a few new notations.
Let us introduce the maximal degree of homogeneity of an index set $\cI$ by
$$\homg{}{\cI}{}:=\max_{I\in\cI}\homg{}{I}{}\qqtext{,}\homg{}{\emptyset}{}:=-\infty\qquad.$$
Moreover, for $\cI\subset\cIq$ and $\cJ\subset\cJqpe$ we define
\begin{align*}
\bDqs(\cI,\om)&:=\big(\Lzqsom\boxplus\eta\calD^q(\cI_s)\big)\cap\bDqlocom\qquad,\\
\bRqpes(\cJ,\om)&:=\big(\Lzqpesom\boxplus\eta\calR^{q+1}(\cJ_s)\big)\cap\bRqpelocom
\end{align*}
and note $\bDqs(\cI,\om)=\bDqsom$ resp. $\bRqpes(\cJ,\om)=\bRqpesom$\,,
if $\cI_s=\emptyset$ resp. $\cJ_s=\emptyset$\,.
From now on we will work with tower-forms of arbitrary heights.
Thus in the following we may assume additionally
$$3\leq N\qqtext{odd}\qquad.$$

We may generalize Theorem \ref{gewstatikinjektivrotdivnull} as described above in the following two lemmas:

\begin{lem}\mylabel{verallstatikdiv}
Let $s\in(1-N/2,\infty)\ohne\pI$ and $\cI$ be a finite subset of $\cIq$ with maximal homogeneity degree $\homg{}{\cI}{}$\,,
such that $\eta\calD^q(\cI)\cap\Lzqsom=\{0\}$ holds true.
Furthermore, let $\tau>\max\{0,s-N/2,s+N/2+\homg{}{\cI}{}\}$ and $\tau\geq-s$\,.
Then for every $q$-form $F\in\bDqs(\cI,\om)$ with
$$F=F_s+\sum_{I\in\cI}{\tt f}_I\cdot\eta D^q_I\qqtext{,}F_s\in\Lzqsom\qqtext{,}{\tt f}_I\in\cz$$
there exists a unique $(q+1)$-form
$$H\in\Big(\big(\mu^\me\ronqpeom{s-1}\cap\dqpeom{s-1}\big)\boxplus\eta\calR^{q+1}(\pcJbqpe{0}{s-1}\cup\cIe)\Big)\cap\mu^\me\bRqpelocom$$
solving $\pdiv H=F$\,. This solution $H$ is represented by
$$H=H_{s-1}+\sum_{J\in\pcJbqpe{0}{s-1}}{\tt g}_J\cdot\eta R^{q+1}_J+\sum_{I\in\cI}{\tt f}_I\cdot\eta R^{q+1}_{{}_1I}$$
with $H_{s-1}\in\mu^\me\ronqpeom{s-1}\cap\dqpeom{s-1}$ and ${\tt g}_J\in\cz$\,. $H$ is an element of $\Lzqpetom$ for
$t<\min\{N/2,-1-N/2-\homg{}{\cI}{}\}$ and $t\leq s-1$\,.
Moreover, the solution operator is continuous and maps in particular
$\bDqs(\cI,\om)$ to $\mu^\me\bR{q+1}{s-1}{0}(\pcJbqpe{0}{s-1}\cup\cIe,\om)$ as well as to $\mu^\me\bRom{q+1}{t}{0}$ continuously.
\end{lem}

\begin{rem}\mylabel{verallstatikdivbem}
Using the notations from the lemma above we obtain by the properties of the order of decay $\tau$
$$\hat{H}:=\mu H\in\Big(\big(\ronqpeom{s-1}\cap\mu\dqpeom{s-1}\big)\boxplus\eta\calR^{q+1}(\pcJbqpe{0}{s-1}\cup\cIe)\Big)\cap\bRqpelocom$$
solving $\pdiv\mu^\me\hat{H}=F$\,. $\hat{H}$ is of the form
$$\hat{H}=\hat{H}_{s-1}+\sum_{J\in\pcJbqpe{0}{s-1}}{\tt g}_J\cdot\eta R^{q+1}_J+\sum_{I\in\cI}{\tt f}_I\cdot\eta R^{q+1}_{{}_1I}$$
with $\bds\hat{H}_{s-1}=\mu H_{s-1}+\sum_{J\in\pcJbqpe{0}{s-1}}{\tt g}_J\cdot\mud\,\eta R^{q+1}_J+\sum_{I\in\cI}{\tt f}_I\cdot\mud\,\eta R^{q+1}_{{}_1I}\in\ronqpeom{s-1}\cap\mu\dqpeom{s-1}\eds$\,.
\end{rem}

\begin{proof}
Let us assume w. l. o. g. $\cI\neq\emptyset$ and
$$F=F_s+\sum_{I\in\cI}{\tt f}_I\cdot\eta D^q_I\in\bDqs(\cI,\om)
=\big(\dqsom\boxplus\eta\calD^q(\cI)\big)\cap\bDqlocom\qquad.$$
By the choice of our cut-off function $\eta$ all terms, which possess a factor $\eta$\,, are perpendicular to
$\bonqom$ resp. $\bqpeom$\,. Especially $\eta D^q_I$ and $F_s$ belong to $\bonqom^\bot$\,.
Noticing $\pdiv R^{q+1}_{{}_1I}=D^q_I$ by \eqref{Turmdrei} we choose the ansatz
$$H:=h+\sum_{I\in\cI}{\tt f}_I\cdot\eta R^{q+1}_{{}_1I}\qquad.$$
Thus our system $H\in\mu^\me\bRqpelocom$ and $\pdiv H=F$ is transformed into
\begin{align*}
\pdiv h&=F-\sum_{I\in\cI}{\tt f}_I\cdot\pdiv(\eta R^{q+1}_{{}_1I})\\
&=F_s-\sum_{I\in\cI}{\tt f}_I\cdot C_{\pdiv,\eta}R^{q+1}_{{}_1I}=:f\in\bDqsom\qquad,\\
\rot\mu h&=-\sum_{I\in\cI}{\tt f}_I\cdot\rot(\mu\,\eta R^{q+1}_{{}_1I})=:g\in\bRom{q+2}{\loc}{0}\qquad,\\
\mu h&\in\bqpeom^\bot\qquad.
\end{align*}
Because of $\tau>s+N/2+\homg{}{\cI}{}\geq s+N/2+\homg{}{I}{}$ Remark \ref{intbarkeitTuerme} yields
$$\rot(\mu\,\eta R^{q+1}_{{}_1I})=C_{\rot,\eta}R^{q+1}_{{}_1I}+\rot(\mud\,\eta R^{q+1}_{{}_1I})\in\qLzom{q+2}{<-\Nh-\homg{}{I}{}+\tau}\subset\qLzom{q+2}{s}$$
for all $I\in\cI$\,. Now we can apply Theorem \ref{gewstatikinjektivrotdivspeziell} and get the unique solution of the system above
$$h:=(\sideset{^\mu}{^{q+1}_{s-1}}\statMax)^\me(f,g,0)\in\big(\mu^\me\ronqpeom{s-1}\cap\dqpeom{s-1}\big)\boxplus\eta\calR^{q+1}(\pcJbqpe{0}{s-1})\boxplus\eta\check{\calR}^{q+1,1}_{s-1}\,\,.$$
Thus $H$ is an element of
$$\Big(\big(\mu^\me\ronqpeom{s-1}\cap\dqpeom{s-1}\big)\boxplus\eta\calR^{q+1}(\pcJbqpe{0}{s-1}\cup\cIe)\boxplus\eta\check{\calR}^{q+1,1}_{s-1}\Big)\cap\mu^\me\bRqpelocom$$
and clearly the desired unique solution because of its special form.
Utilizing the appendix and $H\in\mu^\me\bRqpelocom$ we see that even in the case $q=N-2$ the exceptional form
$\check{R}^{q+1,1}_{s-1}$ does not appear.
\end{proof}

With similar arguments we prove

\begin{lem}\mylabel{verallstatikrot}
Let $s\in(1-N/2,\infty)\ohne\pI$ and $\cJ$ be a finite subset of $\cJqpe$ with maximal homogeneity degree $\homg{}{\cJ}{}$\,,
such that $\eta\calR^{q+1}(\cJ)\cap\Lzqpesom=\{0\}$ holds true.
Furthermore, let $\tau>\max\{0,s-N/2,s+N/2+\homg{}{\cJ}{}\}$ and $\tau\geq-s$\,.
Then for every $(q+1)$-form $G\in\bRqpes(\cJ,\om)$ with
$$G=G_s+\sum_{J\in\cJ}{\tt g}_J\cdot\eta R^{q+1}_J\qqtext{,}G_s\in\Lzqpesom\qqtext{,}{\tt g}_J\in\cz$$
there exists a unique $q$-form
$$E\in\Big(\big(\ronqom{s-1}\cap\eps^\me\dqom{s-1}\big)\boxplus\eta\calD^q(\pcIbq{0}{s-1}\cup\cJe)\Big)\cap\eps^\me\bDqlocom$$
solving $\rot E=G$\,. This solution $E$ is represented by
$$E=E_{s-1}+\sum_{I\in\pcIbq{0}{s-1}}{\tt f}_I\cdot\eta D^q_I+\sum_{J\in\cJ}{\tt g}_J\cdot\eta D^q_{{}_1J}$$
with $E_{s-1}\in\ronqom{s-1}\cap\eps^\me\dqom{s-1}$ and ${\tt f}_I\in\cz$\,. $E$ is an element of $\Lzqtom$ for all
$t<\min\{N/2,-1-N/2-\homg{}{\cJ}{}\}$ and $t\leq s-1$\,.
Moreover, the solution operator is continuous and maps in particular
$\bRqpes(\cJ,\om)$ to $\eps^\me\bD{q}{s-1}{0}(\pcIbq{0}{s-1}\cup\cJe,\om)$ as well as to $\eps^\me\bDom{q}{t}{0}$ continuously.
\end{lem}

\begin{rem}\mylabel{verallstatikrotbem}
Using the notations from the lemma above we obtain by the properties of the order of decay $\tau$
$$\hat{E}:=\eps E\in\Big(\big(\eps\ronqom{s-1}\cap\dqom{s-1}\big)\boxplus\eta\calD^q(\pcIbq{0}{s-1}\cup\cJe)\Big)\cap\bDqlocom$$
solving $\rot\eps^\me\hat{E}=G$\,. $\hat{E}$ has the form
$$\hat{E}=\hat{E}_{s-1}+\sum_{I\in\pcIbq{0}{s-1}}{\tt f}_I\cdot\eta D^q_I+\sum_{J\in\cJ}{\tt g}_J\cdot\eta D^q_{{}_1J}$$
with $\bds\hat{E}_{s-1}=\eps E_{s-1}+\sum_{I\in\pcIbq{0}{s-1}}{\tt f}_I\cdot\epsd\,\eta D^q_I+\sum_{J\in\cJ}{\tt g}_J\cdot\epsd\,\eta D^q_{{}_1J}\in\eps\ronqom{s-1}\cap\dqom{s-1}\eds$\,.
\end{rem}

The latter two lemmas and remarks yield a solution theory for a generalized static Maxwell problem:

\begin{defini}\mylabel{verallstatMaxdef}
Let $s\in(1-N/2,\infty)\ohne\pI$ and $\cI\times\cJ$ be a finite subset of $\cIq\times\cJqpe$\,, such that
$\eta\calD^q(\cI)\cap\Lzqsom=\{0\}$ and $\eta\calR^{q+1}(\cJ)\cap\Lzqpesom=\{0\}$ holds.
Furthermore, let $\tau>\max\{0,s-N/2\}$\,, $\tau\geq-s$ and
$\tau>s+N/2+\max\{\homg{}{\cI}{},\homg{}{\cJ}{}\}$\,.

We call $\EH$ a solution of the `{\sf generalized static Maxwell problem}' for data
$$\FG\in\bDqs(\cI,\om)\times\bRqpes(\cJ,\om)\qquad,$$
if and only if
\begin{align*}
\text{\rm\bf (i)}&&E&\in\Big(\big(\ronqom{s-1}\cap\eps^\me\dqom{s-1}\big)\boxplus\eta\calD^q(\pcIbq{0}{s-1}\cup\cJe)\Big)\cap\eps^\me\bDqlocom\quad,\\
&&H&\in\Big(\big(\mu^\me\ronqpeom{s-1}\cap\dqpeom{s-1}\big)\boxplus\eta\calR^{q+1}(\pcJbqpe{0}{s-1}\cup\cIe)\Big)\cap\mu^\me\bRqpelocom\quad,\\
\text{\rm\bf (ii)}&&&\qquad\qquad\qquad\rot E=G\qqtext{,}\pdiv H=F
\end{align*}
hold.
\end{defini}

We set $\Lambda:=\zmat{\eps}{0}{0}{\mu}$ and obtain

\begin{theo}\mylabel{verallstatMaxsatz}
The generalized static Maxwell problem is always uniquely solvable.
The mapping $\FG\mapsto\EH$ defines two continuous linear operators
\begin{align*}
\loesn\quad:\quad&\quad\bDqs(\cI,\om)\times\bRqpes(\cJ,\om)\\
\To&\quad\Lambda^\me\big(\bD{q}{s-1}{0}(\pcIbq{0}{s-1}\cup\cJe,\om)\times\bR{q+1}{s-1}{0}(\pcJbqpe{0}{s-1}\cup\cIe,\om)\big)
\end{align*}
and $\loes:=\Lambda\loesn$ with $\loesn\FG:=\EH$\,.
\end{theo}

\begin{rem}\mylabel{verallstatMaxbem}
The `tower-parts' of the `generalized static Maxwell operators' can be described more precisely: If
$$F=F_s+\sum_{I\in\cI}{\tt f}_I\cdot\eta D^q_I\qqtext{and}G=G_s+\sum_{J\in\cJ}{\tt g}_J\cdot\eta R^{q+1}_J\qquad,$$
then (for example) the solution $\EH=\loes\FG$ is of the form
$$E=E_{s-1}+\eta\tilde{E}+\sum_{J\in\cJ}{\tt g}_J\cdot\eta D^q_{{}_1J}\qtext{and}H=H_{s-1}+\eta\tilde{H}+\sum_{I\in\cI}{\tt f}_I\cdot\eta R^{q+1}_{{}_1I}\quad,$$
where $(E_{s-1},H_{s-1})\in\Lzqom{s-1}\times\Lzqpeom{s-1}$ and $(\tilde{E},\tilde{H})\in\calD^q(\pcIbq{0}{s-1})\times\calR^{q+1}(\pcJbqpe{0}{s-1})$\,.
\end{rem}

The generalized static Maxwell operator $\loes$ in Theorem \ref{verallstatMaxsatz} may now be iterated easily.
Since the static Maxwell operator \eqref{Mdef} has only entries on its secondary diagonal,
we have to distinguish between even and odd powers of $\loes$\,. We get

\begin{theo}\mylabel{loesit}
Let $j\in\nz$\,, $s\in(j-N/2,\infty)\ohne\pI$ and $\cI\times\cJ$ be a finite subset of $\cIq\times\cJqpe$\,, such that
$\eta\calD^q(\cI)\cap\Lzqsom=\{0\}$ and $\eta\calR^{q+1}(\cJ)\cap\Lzqpesom=\{0\}$ hold. Moreover, let
$\tau>\max\{0,s-N/2\}$\,, $\tau\geq j-1-s$ and $\tau>s+N/2+\max\{\homg{}{\cI}{},\homg{}{\cJ}{}\}$\,.
Then
\begin{align*}
\loes^j\quad:\quad&\bDqs(\cI,\om)\times\bRqpes(\cJ,\om)\\
\To&
\begin{cases}
\bD{q}{s-j}{0}(\pcIbq{\leq j-1}{s-j}\cup\cIj,\om)\times\bR{q+1}{s-j}{0}(\pcJbqpe{\leq j-1}{s-j}\cup\cJj,\om)&\,,j\text{ even}\\
&\\
\bD{q}{s-j}{0}(\pcIbq{\leq j-1}{s-j}\cup\cJj,\om)\times\bR{q+1}{s-j}{0}(\pcJbqpe{\leq j-1}{s-j}\cup\cIj,\om)&\,,j\text{ odd}
\end{cases}
\end{align*}
is a continuous linear operator, whose range is contained in
$$\bDom{q}{t}{0}\times\bRom{q+1}{t}{0}$$
for all $t$ satisfying $t\leq s-j$\,, $t<N/2-j+1$ and $t<-j-N/2-\max\{\homg{}{\cI}{},\homg{}{\cJ}{}\}$\,.
\end{theo}

\begin{rem}\mylabel{loesitbem}
Also for higher powers $\loes^j$ of $\loes$ it is clear by Remark \ref{verallstatMaxbem}, in which way $\loes^j$ maps
tower-forms to tower-forms. Furthermore, this remark shows that the new appearing tower-forms from
$\eta\calD^q(\pcIbq{\leq j-1}{s-j})$ and $\eta\calR^{q+1}(\pcJbqpe{\leq j-1}{s-j})$ satisfy the following recursion:
Let $\FG$ be as in Remark \ref{verallstatMaxbem}. If $\EH:=\loes^j\FG$ has the form
\begin{align*}
\EH&=(E_{s-j},H_{s-j})+\Big(\sum_{I\in\pcIbq{\leq j-1}{s-j}}{\tt e}_I\cdot\eta D^q_I,\sum_{J\in\pcJbqpe{\leq j-1}{s-j}}{\tt h}_J\cdot\eta R^{q+1}_J\Big)\\
&\qquad\qquad+\begin{cases}
\bds\Big(\sum_{I\in\cI}{\tt f}_I\cdot\eta D^q_{{}_jI},\sum_{J\in\cJ}{\tt g}_J\cdot\eta R^{q+1}_{{}_jJ}\Big)\eds&\,,j\text{ even}\\
&\\
\bds\Big(\sum_{J\in\cJ}{\tt g}_J\cdot\eta D^q_{{}_jJ},\sum_{I\in\cI}{\tt f}_I\cdot\eta R^{q+1}_{{}_jI}\Big)\eds&\,,j\text{ odd}
\end{cases}\qquad,
\end{align*}
where $(E_{s-j},H_{s-j})\in\Lzqom{s-j}\times\Lzqpeom{s-j}$\,, then
\begin{align*}
(\tilde{E},\tilde{H})&:=\loes\EH=\loes^{j+1}\FG\\
&\,\,=(\tilde{E}_{s-j-1},\tilde{H}_{s-j-1})+\Big(\sum_{I\in\pcIbq{\leq j}{s-j-1}}\tilde{\tt e}_I\cdot\eta D^q_I,\sum_{J\in\pcJbqpe{\leq j}{s-j-1}}\tilde{\tt h}_J\cdot\eta R^{q+1}_J\Big)\\
&\qquad\qquad+\begin{cases}
\bds\Big(\sum_{J\in\cJ}{\tt g}_J\cdot\eta D^q_{{}_{j+1}J},\sum_{I\in\cI}{\tt f}_I\cdot\eta R^{q+1}_{{}_{j+1}I}\Big)\eds&\,,j\text{ even}\\
&\\
\bds\Big(\sum_{I\in\cI}{\tt f}_I\cdot\eta D^q_{{}_{j+1}I},\sum_{J\in\cJ}{\tt g}_J\cdot\eta R^{q+1}_{{}_{j+1}J}\Big)\eds&\,,j\text{ odd}
\end{cases}\qquad,
\end{align*}
where $(\tilde{E}_{s-j-1},\tilde{H}_{s-j-1})\in\Lzqom{s-j-1}\times\Lzqpeom{s-j-1}$\,.
Thereby, for indices $I\in\pcIbq{\leq j-1}{s-j}$ and $J\in\pcJbqpe{\leq j-1}{s-j}$ the coefficients ${\tt e}_I$\,, ${\tt h}_J$ and $\tilde{\tt e}_{{}_1J}$\,, $\tilde{\tt h}_{{}_1I}$ satisfy the recursion
$${\tt e}_I=\tilde{\tt h}_{{}_1I}\qqtext{,}{\tt h}_J=\tilde{\tt e}_{{}_1J}\qquad.$$
\end{rem}

Finally we formulate the latter theorem in the special case $\cI=\emptyset$\,, $\cJ=\emptyset$\,:

\begin{cor}\mylabel{loesitkor}
Let $j\in\nz$\,, $s\in(j-N/2,\infty)\ohne\pI$ and $t\leq s-j$\,, $t<N/2-j+1$ as well as
$\tau>\max\{0,s-N/2\}$\,, $\tau\geq j-1-s$\,. Then
$$\loes^j:\bDqsom\times\bRqpesom\To\bD{q}{s-j}{0}(\pcIbq{\leq j-1}{s-j},\om)\times\bR{q+1}{s-j}{0}(\pcJbqpe{\leq j-1}{s-j},\om)$$
is a continuous linear operator with range contained in $\bDom{q}{t}{0}\times\bRom{q+1}{t}{0}$\,.
\end{cor}

\section{Electro-magneto statics with inhomogeneous boundary data}

We want to conclude this paper by discussing inhomogeneous boundary data.\mylabel{inhomo}
Recently Weck showed in \cite{wecklip}, how one may obtain traces of differential forms on Lipschitz-boundaries.
To utilize his results let us assume that $\om$ has a Lipschitz boundary.
As in \paulytimeharmsecboundary\, we then have
for every $s\in\rz$ linear and continuous tangential trace and extension operators $\gt$ and $\chgt$ satisfying
$$\rqsom\xrightarrow{\gt}\xxrq(\dom)\xrightarrow{\chgt}\rqsom\cap\eps^\me\dqsom$$
and $\gt\chgt=\id$ on $\xxrq(\dom)$\,. We note that the kernel of $\gt$ equals $\ronqsom$ and
that $\gt$ may be defined even on $\rqlocomq$\,. $\chgt$ may be chosen, such that
$\supp\chgt\lambda\subset\ol{\om\cap U_{r_2}}$ holds for all $\lambda\in\xxrq(\dom)$\,,
in particular $\chgt$ maps to $\rqvoxom\cap\eps^\me\dqvoxom$\,.

Now our aim is to generalize the static solution theory, such that we can deal with inhomogeneous boundary data.

With the functionals $\Phi^\ell_\eps$ used in Theorem \ref{gewstatikinjektivrotdiv} we consider the following problem:
For some given data $G,f,\lambda,\alpha$ find a $q$-form $E\in\rqom{>-\Nh}\cap\eps^\me\dqom{>-\Nh}$ satisfying
\begin{align}
\begin{split}
\rot E&=G\qquad,\\
\pdiv\eps E&=f\qquad,\\
\gt E&=\lambda\qquad,\\
\Phi^\ell_\eps(E)&=\alpha_\ell\qqtext{,}\ell=1,\dots,d^q\qquad.
\end{split}\mylabel{problem}
\end{align}

\begin{theo}\mylabel{loesstatikgewichtrand}
Let $s\in(1-N/2,\infty)\ohne\pI$ and $\tau>\max\{0,s-N/2\}$\,, $\tau\geq-s$\,.
Then for all $\alpha\in\cz^{d^q}$\,,
$f\in\bDom{q-1}{s}{0}$ and all $G\in\rqpesnom$\,, $\lambda\in\xxrq(\dom)$ satisfying
$$\Rot\lambda=\gt G\qquad\wedge\qquad\bigwedge_{h\in\dhqpeom}\skp{G}{h}_{\Lzqpeom{}}=\skp{\rot\chgt\lambda}{h}_{\Lzqpeom{}}$$
there exists a unique solution
$$E\in\big(\rqom{s-1}\cap\eps^\me\dqom{s-1}\big)\oplus\eta\calD^q(\pcIbq{0}{s-1})\oplus\eta\check{\calD}^{q,1}_{s-1}$$
of \eqref{problem}. The solution depends continuously on the data.
\end{theo}

\begin{proof}
With $\check{E}:=\chgt\lambda\in\rqvoxom\cap\eps^\me\dqvoxom$ the ansatz $E:=\check{E}+\tilde{E}$
leads us to the following problem:
Find some $\tilde{E}\in\big(\ronqom{s-1}\cap\eps^\me\dqom{s-1}\big)\oplus\eta\calD^q(\pcIbq{0}{s-1})\oplus\eta\check{\calD}^{q,1}_{s-1}$
solving the system
\begin{align*}
\rot\tilde{E}&=G-\rot\check{E}=:\tilde{G}\in\rqpenom{s}\qquad,\\
\pdiv\eps\tilde{E}&=f-\pdiv\eps\check{E}=:\tilde{f}\in\bDom{q-1}{s}{0}\qquad,\\
\Phi^\ell_\eps(\tilde{E})&=\alpha_\ell-\Phi^\ell_\eps(\check{E})=:\tilde{\alpha}_\ell\qqtext{,}\ell=1,\dots,d^q\qquad.
\end{align*}
By Theorem \ref{gewstatikinjektivrotdiv} this problem is uniquely solved by
$\tilde{E}:=(\sideset{_\eps}{^q_{s-1}}\MAX)^\me(\tilde{f},\tilde{G},\tilde{\alpha}_\ell)$\,,
if $(\tilde{f},\tilde{G},\tilde{\alpha}_\ell)\in \bWqsom$ holds. So it remains to show
$$\tilde{G}\in\bRqpesom=\ronqpenom{s}\cap\dhqpeom^\bot\qquad.$$
Since $\gt\rot=\Rot\gt$ holds,
where $\Rot:=\pd$ denotes the exterior derivative on the submanifold $\dom$ of $\omq$\,,
$\tilde{G}$ satisfies the homogeneous boundary condition and
clearly $\tilde{G}$ is orthogonal to all Dirichlet forms.
\end{proof}

\begin{rem}\mylabel{loesstatikgewichtrandremone}
The orthogonality constraints on the Dirichlet forms may be replaced by constraints on the special forms $\bonqom$ resp. $\bqom$
as in section \ref{statics}. For this it is necessary that the forms $\bonqom$ are irrotational and $\bqom$ solenoidal.
Similarly to section \ref{statics} we are also able to specialize the functionals $\Phi^\ell_\eps$ using $\bonqom$ and $\bqom$\,.
\end{rem}

\begin{rem}\mylabel{loesstatikgewichtrandremtwo}
Clearly we get as well a generalized static solution theory in the case of inhomogeneous boundary data, which acts on arbitrary tower-forms
as in section \ref{iteration}. Then even for inhomogeneous boundary data the iteration process
from section \ref{iteration} holds true in a canonical way. We note that the inhomogeneous boundary condition is only realized by the trace of the
form from the ground floor. All forms from higher floors have vanishing boundary traces.
\end{rem}

\begin{rem}\mylabel{loesstatikgewichtrandremthree}
Assuming more regularity of $\om$\,, i.e. $\om\in\pc{2}$\,, we have by Stokes theorem
$$\skp{\rot\chgt\lambda}{h}_{\Lzqpeom{}}=\skp{\lambda}{\xgn h}_{\qH{\meh}{q}{}{}(\dom)}\qquad,$$
because then $h$ is an element of $\qHom{1}{q+1}{}{}$ and thus $\xgn h$ belongs to $\qH{1/2}{q}{}{}(\dom)$\,,
where $\xgn=\pm\cast\iota^**$ denotes the usual normal trace.
Here $\cast$ denotes the star-operator on the manifold $\dom$\,,
$\iota^*$ the pullback of the natural embedding $\iota:\dom\to\omq$
and $\bds\skp{\,\cdot\,}{\,\cdot\,}_{\qH{\meh}{q}{}{}(\dom)}\eds$
the duality between $\qH{\meh}{q}{}{}(\dom)$ and $\qH{\frac{1}{2}}{q}{}{}(\dom)$\,.
\end{rem}

\appendix

\section{Appendix: Second order operators}

We still have to exclude the appearance of the special tower-forms $\eta\check{D}^{q,1}_{s-1}$\,, i.e. $\eta\check{D}^{1,1}_{s-1}$\,,
and $\eta\check{R}^{q+1,1}_{s-1}$\,, i.e. $\eta\check{R}^{N-1,1}_{s-1}$\,, in Theorem \ref{gewstatikinjektivrotdivnull}, Lemma \ref{verallstatikdiv} and Lemma \ref{verallstatikrot}.
To prove this we introduce a second order approach to our static systems and use once more the relationship between
Maxwell equations and the Poisson equation via the well known formula 
$$\Delta=\rot\pdiv+\pdiv\rot\qquad.$$
Let us introduce the Hilbert spaces
\begin{align*}
\X^q_s(\Omega)&:=\setb{E\in\ronqsom\cap\dqsom}{\mu^\me\rot E\in\dqpeom{s+1}}\qquad,\\
\Y^{q+1}_s(\Omega)&:=\setb{H\in\ronqpesom\cap\dqpesom}{\eps^\me\pdiv H\in\ronqom{s+1}}
\end{align*}
and mention the following fact: If $\eps$ resp. $\mu$ is a $\tau$-$\pc{1}$-admissible transformation on
$q$- resp. $(q+1)$-forms, then so is the inverse transformation $\eps^\me$ resp. $\mu^\me$\,.

The following lemmas can be proved using the same ideas and techniques, which we have presented in sections \ref{statopweight} and \ref{iteration}
for our first order Maxwell systems. We neglect the (very similar) proofs and refer the interested reader to {\paulydissabschnittsechsfuenf}.

\begin{lem}\mylabel{loestheozweiterOdivrot}
Let $s\in[1,\infty)\ohne\pI$ and $\tau>\max\{0,s-N/2\}$\,. Then
\begin{align*}
&\Abb{{}_{\pdiv}\Delta^q_{s-2}}{D({}_{\pdiv}\Delta^q_{s-2})}{\bDqsom}{E}{\pdiv\mu^\me\rot E}\qquad,\\
&\Abb{{}_{\rot}\Delta^{q+1}_{s-2}}{D({}_{\rot}\Delta^{q+1}_{s-2})}{\bRqpesom}{H}{\rot\eps^\me\pdiv H}
\end{align*}
are continuous and surjective Fredholm operators on their domains of definition
\begin{align*}
D({}_{\pdiv}\Delta^q_{s-2})&:=\big(\X^q_{s-2}(\Omega)\boxplus\eta\calD^q(\pcIbq{\leq1}{s-2})\boxplus\eta\check{\calD}^{q,2}_{s-2}\big)\cap\dqlocnom\qquad,\\
D({}_{\rot}\Delta^{q+1}_{s-2})&:=\big(\Y^{q+1}_{s-2}(\Omega)\boxplus\eta\calR^{q+1}(\pcJbqpe{\leq1}{s-2})\boxplus\eta\check{\calR}^{q+1,2}_{s-2}\big)\cap\ronqpelocnom
\end{align*}
with kernels $N({}_{\pdiv}\Delta^q_{s-2})=\dhqom$ and $N({}_{\rot}\Delta^{q+1}_{s-2})=\dhqpeom$\,.
\end{lem}

\begin{lem}\mylabel{keineAusnahmeformenF}
Let $s\in[1,\infty)\ohne\pI$ and $\cI$ be a finite subset of
the index set $\cIq$ with maximal degree of homogeneity $\homg{}{\cI}{}$\,, such that
$\eta\calD^q(\cI)\cap\Lzqsom=\{0\}$ holds. Furthermore, let $\tau>\max\{0,s-N/2\}$\,, $\tau\geq-s$
and $\tau>s+N/2+\homg{}{\cI}{}$\,.
Then for every form $F\in\bDqs(\cI,\om)$ with
$$F=F_s+\sum_{I\in\cI}{\tt f}_I\cdot\eta D^q_I\qqtext{,}F_s\in\Lzqsom\qqtext{,}{\tt f}_I\in\cz$$
there exists a form
$$E\in\X^q_{s-2}(\Omega)\boxplus\eta\calD^q(\pcIbq{\leq1}{s-2}\cup\cIz)\boxplus\eta\check{\calD}^{q,2}_{s-2}$$
solving $\pdiv\mu^\me\rot E=F$\,. Such an $E$ may be represented by
$$E=E_{s-2}+\tilde{E}+\sum_{I\in\cI}{\tt f}_I\cdot\eta D^q_{{}_2I}\qquad,$$
where $E_{s-2}\in\X^q_{s-2}(\Omega)$ and $\tilde{E}\in\eta\calD^q(\pcIbq{\leq1}{s-2})\boxplus\eta\check{\calD}^{q,2}_{s-2}$\,.
\end{lem}

\begin{lem}\mylabel{keineAusnahmeformenG}
Let $q\neq0$\,, $s\in[1,\infty)\ohne\pI$ and $\cJ$ be a finite subset of $\cJqpe$ with maximal degree of homogeneity $\homg{}{\cJ}{}$\,, such that
$\eta\calR^{q+1}(\cJ)\cap\Lzqpesom=\{0\}$ holds. Furthermore, let $\tau>\max\{0,s-N/2\}$\,, $\tau\geq-s$
and $\tau>s+N/2+\homg{}{\cJ}{}$\,.
Then for every form $G\in\bRqpes(\cJ,\om)$ with
$$G=G_s+\sum_{J\in\cJ}{\tt g}_J\cdot\eta R^{q+1}_J\qqtext{,}G_s\in\Lzqpesom\qqtext{,}{\tt g}_J\in\cz$$
there exists a form
$$H\in\Y^{q+1}_{s-2}(\Omega)\boxplus\eta\calR^{q+1}(\pcJbqpe{\leq1}{s-2}\cup\cJz)\boxplus\eta\check{\calR}^{q+1,2}_{s-2}$$
solving $\rot\eps^\me\pdiv H=G$\,. Such a $H$ may be represented by
$$H=H_{s-2}+\tilde{H}+\sum_{J\in\cJ}{\tt g}_J\cdot\eta R^{q+1}_{{}_2J}\qquad,$$
where $H_{s-2}\in\Y^{q+1}_{s-2}(\Omega)$ and $\tilde{H}\in\eta\calR^{q+1}(\pcJbqpe{\leq1}{s-2})\boxplus\eta\check{\calR}^{q+1,2}_{s-2}$\,.
\end{lem}

Now we can show easily that the special forms $\eta\check{D}^{1,1}_{s-1}$ and $\eta\check{R}^{N-1,1}_{s-1}$ do not appear
in Theorem \ref{gewstatikinjektivrotdivnull}, Lemma \ref{verallstatikdiv} or Lemma \ref{verallstatikrot} (in the cases
$q=1$ and $q=N-2$). Since these forms can only occur for weights $s\geq N/2$\,,
we can apply the latter two lemmas (for these $s$) getting some $\EHs$ and obtain the unique solutions $E$ of $\rot E=G$ and $H$ of $\pdiv H=F$ by
\begin{align*}
E&:=\eps^\me\pdiv\tilde{H}\\
&\,\in\,\Big(\big(\ronqom{s-1}\cap\eps^\me\dqom{s-1}\big)\boxplus\eta\calD^q(\pcIbq{0}{s-1}\cup\cJe)\Big)\cap\eps^\me\bDqlocom\qquad,\\
H&:=\mu^\me\rot\tilde{E}\\
&\,\in\,\Big(\big(\mu^\me\ronqpeom{s-1}\cap\dqpeom{s-1}\big)\boxplus\eta\calR^{q+1}(\pcJbqpe{0}{s-1}\cup\cIe)\Big)\cap\mu^\me\bRqpelocom\qquad.
\end{align*}

\begin{acknow}
This research was supported by the {\it Deutsche Forschungsgemeinschaft}
via the project {\sf `We 2394: Untersuchungen der Spektralschar verallgemeinerter
Maxwell-Operatoren in unbeschr\"ankten Gebieten'}.

The author is particularly indebted to his academic teachers Norbert Weck and Karl-Josef Witsch 
for introducing him to the field.
\end{acknow}

\end{document}